\documentclass[11pt]{article}
\usepackage{graphicx}
\usepackage{amsmath}
\usepackage{amsfonts}
\usepackage[margin=1in]{geometry}
\usepackage{amsthm}
\usepackage{setspace,enumerate}
\usepackage{physics}
\usepackage{color}

\newtheorem{theorem}{Theorem}[section]
\newtheorem{lemma}[theorem]{Lemma}
\newtheorem{proposition}[theorem]{Proposition}

\newtheorem{definition}[theorem]{Definition}
\newtheorem{remark}[theorem]{Remark}

\DeclareMathOperator*{\esssup}{ess\,sup}
\DeclareMathOperator*{\essinf}{ess\,inf}

\title{Optimal design of a model energy conversion device}
\author{Lincoln Collins and Kaushik Bhattacharya\\
\small Division of Engineering and Applied Science\\
\small California Institute of Technology\\
\small Pasadena, CA 91125\\
\small Email: lcollins@caltech.edu and bhatta@caltech.edu}
\begin{document}

\maketitle
\begin{abstract}
Fuel cells, batteries, thermochemical and other energy conversion devices involve the transport of a number of (electro-)chemical species through distinct materials so that they can meet and react at specified multi-material interfaces.  Therefore, morphology or arrangement of these different materials can be critical in the performance of an energy conversion device.  In this paper, we study a model problem motivated by a solar-driven thermochemical conversion device that splits water into hydrogen and oxygen.  We formulate the problem as a system of coupled multi-material reaction-diffusion equations where each species diffuses selectively through a given material and where the reaction occurs at multi-material interfaces.  We express the problem of optimal design of the material arrangement as a saddle point problem and obtain an effective functional which shows that regions with very fine phase mixtures of the material arise naturally.  To explore this further, we introduce a phase-field formulation of the optimal design problem, and numerically study selected examples.
\end{abstract}

%%%%%%%%%%%%%%%%%%%%%%%%%%%%%
\section{Introduction}
The efficiency of fuel cells, batteries and thermochemical energy conversion devices depends on inherent material characteristics that govern the complex chemistry and transport of multiple species as well as the spatial arrangement of the various materials. Therefore, optimization of the spatial arrangement is a recurrent theme in energy conversion devices.  Traditional methods of synthesis offer limited control of the microstructure and there has been much work in advanced imaging for these uncontrolled microstructures (e.g., \cite{barnett}) and optimizing gross features.  However, the growing ability for directed synthesis \cite{umeda,jung,wen,li} allows us to ask the question of what microgeometries are optimal for particular applications. In this sense we direct the problem to one of optimal design where we are not limited by the imagination in determining new microstructures but instead allow for the underlying physical behavior and optimization techniques to direct architecture and microstructure, and eventually lead synthesis to unprecedented performance. 

The tailoring of material microstructure and nanostructure is not new to energy conversion and storage \cite{li,arico}.  The development of hierarchical structures and porosity affords balancing interfacial reactions and chemical transport to maximize efficiency.  For example, in cathodes of solid oxide fuel cells, the efficiency is largely determined by reactions at triple phase boundaries and the availability of transport pathways through each phase.  These features are coupled with volume fractions, surface area densities, interfacial curvatures, and phase tortuosities to find the optimal balance between surface reaction and transport \cite{wilson,smith}.  Similarly, the importance of microstructure on anode performance \cite{atkinson,suzuki,cronin,clemmer} has also been established.  The morphology of materials used in lithium ion batteries is of interest from both the theoretical \cite{stephenson,chen2010porous,bruce} and experimental standpoint \cite{hu,jia,wen}.  Mass and ion transport and interface measure in battery electrodes directly impact the storage capacity and rate performance and is an ideal problem for optimization across many length scales.

The application of metal oxides for solar-driven thermochemical conversion devices offers a promising new sustainable energy source \cite{Chueh}.  Here, a porous, redox active oxide is cyclically exposed to inert gas at high temperature, generating oxygen vacancies in the structure, and reactant gas ($\text{H}_2\text{O},\text{ CO}_2$), at moderate temperature, releasing fuel upon reoxidation the oxide \cite{umeda}.  The lack of complex and expensive catalyst systems and full use of the entire solar spectrum separate these devices from many other photo-based energy sources.  Recent advances made in the materials research community indicate many possible candidates for these applications, and lend themselves to advanced synthesis techniques facilitating directed architecture, where significant improvements can be made \cite{venstrom}.  The thermodynamic and kinetic behavior of these materials are well-studied \cite{lai,Chuehb,gopal}, fully describing the gas phase transport of reactant and product gases coupled to the solid state mass and charge transport occurring through the bulk.   

In this paper, we study a model system motivated by metal oxides in solar-driven thermochemical conversion devices.  We have a two phase material (solid oxide and pore) where reactions at the surface create (gaseous) oxygen in the carrier gas in the pores and bound oxygen in the solid oxide; the oxygen diffuses through the carrier gas in the porous region and bound oxygen diffuses through the solid oxide.  We seek to understand the arrangement of the solid and porous regions to maximize the transport given sources and sinks for the gaseous oxygen and vacancies.  

There is  a large literature in the study of optimal design problems, especially seeking to minimize compliance for a given weight as well as maximize conduction for a given mass.  It is understood that the underlying problem is ill-posed in that the optimal designs often lie outside of the set of ``classical admissible designs" and one has to either relax the problem by homogenization \cite{Goodman, kohnstrang1} or regularize it by the introduction of perimeter constraints  \cite{Ambrosio1993,bourdin}.  It leads to two widely used methods, topology optimization  (e.g. \cite{bendsoetopology}) and shape optimization (e.g. \cite{allaire}).  The presence of two species lends a vectorial character to our problem, and the presence of the surface sources makes the problem at hand different from those in the literature.

%to study this problem.  Topology optimization offers a framework for designing structures with optimal response given a set of constraints and material behavior.   Topology optimization and the closely related technique of shape optimization has been extensively used in designing structure that minimize compliance for a given weight as well as maximize conduction for a given mass.  The mathematical structure of these problems have been well-studied.  Typical problems in material arrangement are known to be.  These have led to good numerical methods to solve specific problems (e.g. \cite{Bendsoe1999, allaire, kohnstrang1}).

We begin with the formulation of physical problem, shown in Figure \ref{fig:domain}, in Section \ref{sec:formulation}.  We start with a sharp interface formulation.  However, the optimal design of the sharp interface model is mathematically ill-posed, and therefore we study the analogous diffuse interface model.  We also note that homogenization of the sharp interface model leads to equations of the same form as the diffuse interface equations.  The transport of two chemical species with concentration $u_1$ and $u_2$ is governed by the following reaction diffusion equations for $i=1,2$:
\begin{align} \label{eq:diffuse}
\begin{cases}
\nabla\cdot k_i \nabla u_i = f_i , \quad & \text{in}\;\Omega, \\
k_i \nabla u_i\cdot \hat{n}=0& \text{on}\;\partial\Omega\setminus\partial_{i}\Omega\\
u_i=u_i^*&\text{on}\;\partial_{i}\Omega,
\end{cases}
\end{align}
where the isotropic conductivities are 
\begin{equation} \label{eq:cond}
k_1(x)= k_{11} \chi(x)+ k_{12} (1- \chi(x)) , \quad
k_2(x)=k_{21}\chi(x) + k_{22} (1- \chi(x)) ;
\end{equation}
with $k_{11}, k_{22} >> k_{12},k_{21}  >0$, and the sources are
\begin{equation} \label{eq:source}
f_1 = - f_2 = \chi (1-\chi)  k_s (u_1 - u_2)
\end{equation}
for $k_s>0$ and $\chi: \Omega \to [0,1]$.  Briefly, we have a two-material system and $\chi$ describes the volume fraction of material 1 (say solid phase).  Chemical species 1 (say bound oxygen) diffuses preferentially in material 1 ($\chi =1$) while species 2 (say oxygen gas) diffuses preferentially in material 2 (say pore, $\chi =0$).  The species react and therefore there is a source at the interface $\chi \ne 0,1$.  For future use, we write the source as $f = \chi (1-\chi)  Au$ where $f = \{f_1, f_2\}$, $u=\{u_1,u_2\}$ and 
$$
A=k_s\left(\begin{array}{cc}
1 & -1\\
-1 & 1
\end{array}\right).
$$ 
This problem allows a variational formulation, and the direct method of the calculus of variation allows us to prove existence of a solution.

We study the optimal design problem of maximizing the flux of species through the reactor over all possible arrangements $\chi$ in Section \ref{sec:opt}.  We show that this gives rise to a saddle point problem.  We then obtain an explicit characterization which shows that the mixed phase regions arise naturally.  To understand this further through particular examples, we introduce a phase field formulation in Section \ref{sec:phasefield}.  Specifically, we add an Allen-Cahn type energy to that associated with the variational formulation of the transport problem and then solve the gradient flow associated with this energy.  We solve this numerically in selected examples and conduct a parameter study.  These show that the optimal design can be quite intricate as it seeks to balance transport and reaction.

%%%%%%%%%%%%%%%%%%%%%%%%%%%%%
\section{Formulation} \label{sec:formulation}

%%%%%%%%%%%%%%%%%%%%%%%%%%%%%
\subsection{Sharp interface formulation} \label{sec:sharp}
\begin{figure}\label{fig:domain}
\centering
\includegraphics[width=0.26\textwidth]{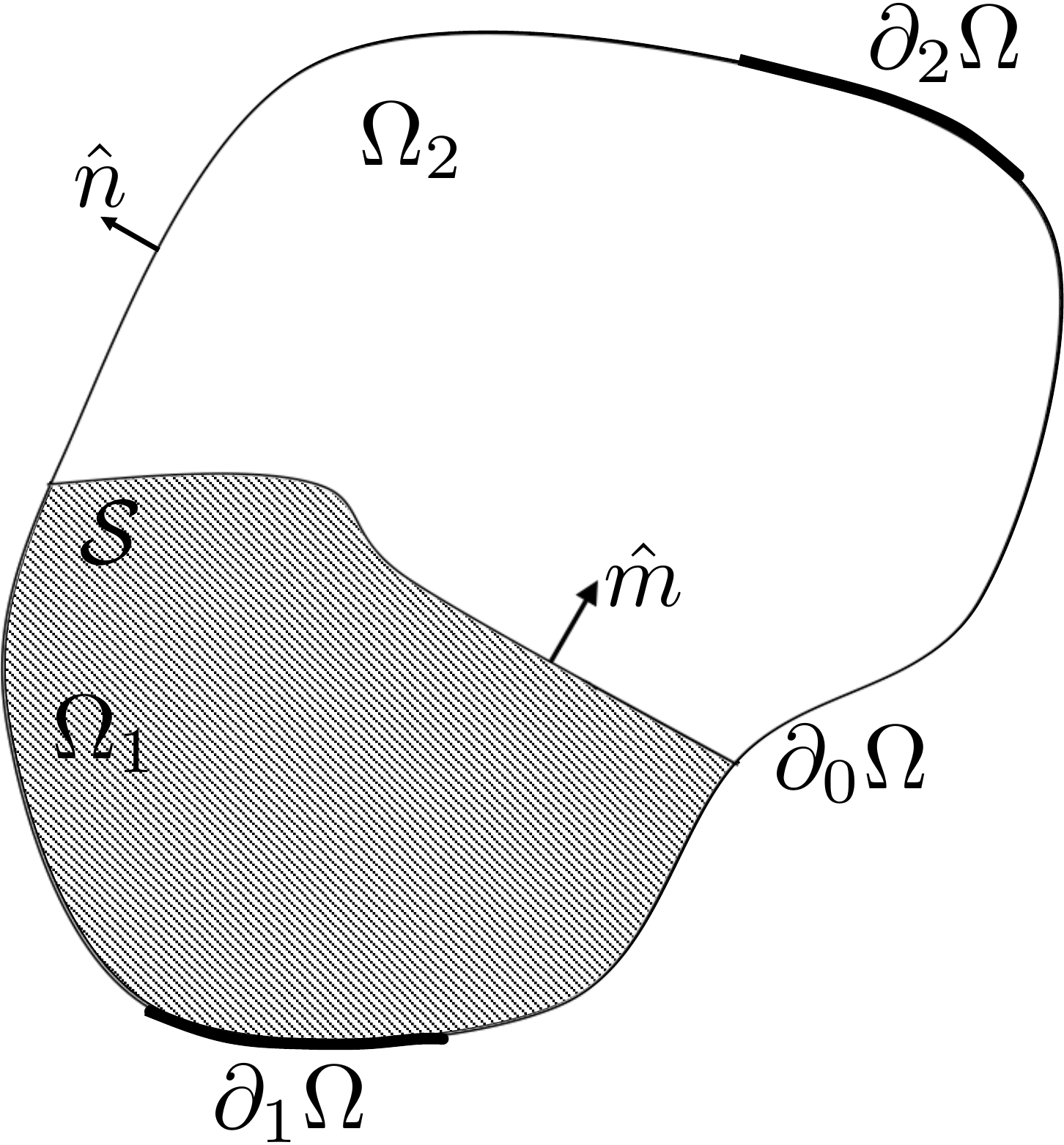}
\caption{The physical setting: chemical species 1 enters through the source $\partial_1\Omega$, diffuses through $\Omega_1$, is converted to chemical species 2 through a surface reaction at the interface $\mathcal{S}$, chemical species 2 diffuses through $\Omega_2$ and leaves through the sink $\partial_2\Omega$.}
\end{figure}

Consider an open, bounded region $\Omega \subset {\mathbb R}^n$ with Lipschitz boundary separated into two regions $\Omega_1$ and $\Omega_2$ by an interface $\mathcal{S}$ show in Figure \ref{fig:domain}.   We consider the diffusion of one species with concentration $u_1$ in region $\Omega_1$ with isotropic diffusivity $K_1>0$, and a second species with concentration $u_2$ in region $\Omega_2$ with isotropic diffusivity $K_2>0$.  The two species meet at the interface and react with reaction rate $k_s >0$.  The boundary of $\Omega$ is divided into three regions $\partial \Omega = \partial_1 \Omega \cup  \partial_2 \Omega \cup  \partial_0 \Omega$ where $\partial_i \Omega \subset \partial \Omega_i$.  The concentration of species $i$ is held at a prescribed value $u_i^*$ on $\partial_i \Omega$ while $\partial_0 \Omega$ is insulating.  
This is described by the following system of equations:
\begin{align} \label{eq:sharp}
\begin{cases}
\nabla\cdot K_i \nabla u_i=0 \quad &\text{in}\;\Omega_i\\
-K_i\nabla u_i\cdot \hat{m}=k_s(u_1-u_2) &\text{on}\;\mathcal{S}\\
u_i=u_i^* &\text{on}\;\partial_{i}\Omega\\
K_i\nabla u_i\cdot \hat{n}=0&\text{on}\;\partial\Omega\setminus\partial_i\Omega
\end{cases}
\end{align}
for $i=1,2$ where $\hat{m}$ represents the normal to $\mathcal{S}$ pointing from $\Omega_1$ pointing to $\Omega_2$, and $\hat{n}$ represents the outward normal to $\partial\Omega$.
\subsection{Diffuse interface formulation}

It is often convenient to work with a smooth or diffuse interface formulation of the problem above.  We now show formally that the diffuse interface formulation in (\ref{eq:diffuse})-(\ref{eq:source}) leads to the sharp interface formulation in (\ref{eq:sharp}) in an asymptotic limit.  Let $\chi$ be the characteristic function of $\Omega_1$ as defined in Section \ref{sec:sharp}.  Let $\chi^\eta$ be the mollification of $\chi$ with a standard mollifier at length-scale $\eta$: $\chi^\eta = \varphi^\eta * \chi$ where $\varphi^\eta (x) = \eta^{-n} \varphi(x/\eta)$.  Let $k_1^\eta, k_2^\eta$ be as in (\ref{eq:cond}) with $\chi = \chi^\eta$ and $k_{12}^\eta = (1-\exp(-\eta)) k_{12}$, $k_{21}^\eta = (1-\exp(-\eta)) k_{21}$, and  $f_i^\eta (x) = f_i(x/\eta)$.  Let $u_i^\eta$ solve
\begin{align} \label{eq:formal}
\nabla\cdot k_i^\eta \nabla u_i^\eta = \eta^2 f_i^\eta , \quad i = 1,2 \quad \text{in}\;\Omega.
\end{align}
First consider the outer expansion $\eta \to 0$, and note that (\ref{eq:formal}) formally gives (\ref{eq:sharp})$_{1,3,4}$ in $\Omega_{1,2}$.  Further, note that $u_1$ (respectively $u_2$) is indeterminate on $\Omega_2$ (respectively $\Omega_1$).  However, this outer expansion does not give any condition on the interface $\mathcal{S}$. To obtain this condition, denote the limiting values on the interface to be $\bar{u}_1, \bar{u}_2$.  We seek to relate these to the flux as in (\ref{eq:sharp})$_{2}$.  

Now consider the inner expansion.  Pick a point $x_0 \in {\mathcal S}$ and change variables $x \mapsto (x-x_0)/\eta$.  We obtain 
\begin{align} \label{eq:formalinner}
\nabla\cdot k_i \nabla u_i =  f_i , \quad i = 1,2
\end{align}
where $k_i = k_i^1$.  Further, as $\eta \to 0$, $\chi$ and hence the solution depend only on one dimension that is normal to the interface.  We take this direction to be $x_1$ by changing variables if necessary.  Let $U_i$ solve (\ref{eq:formalinner}) for the boundary conditions $(u_1, u_2) (x_1) \to (1,0)$ as $x_1 \to -\infty$ and $(u_1, u_2) (x_1) \to (0,1)$ as $x_1 \to \infty$, and $V_i$ solve (\ref{eq:formalinner}) for the boundary conditions $(u_1, u_2) (x_1) \to (1,0)$ as $x_1 \to -\infty$ and $(u_1, u_2) (x_1) \to (0,-1)$ as $x_1 \to \infty$.  Note that 
$$
u_i = \alpha U_i + \beta V_i + \gamma
$$
also solves (\ref{eq:formalinner}) for any arbitrary scalars $\alpha, \beta, \gamma$, 
and satisfies the boundary conditions 
$$
u_1 \to \alpha + \beta + \gamma \mbox { as } x_1 \to -\infty, \quad 
u_2 \to \alpha - \beta + \gamma \mbox{ as } x_1 \to \infty.
$$
Further, by integrating (\ref{eq:formalinner}), we find that the flux 
$$
J = [[-k_i \nabla u_1\cdot e_1]]_{-\infty}^{\infty} =  K_1 u_1' (-\infty)
=  [[k_i \nabla u_2\cdot e_1]]_{-\infty}^{\infty} = - K_2 u_2' (+\infty)
= \alpha J_U + \beta J_V
$$
where $J_U$, $J_V$ are the fluxes associated with the solutions $U$ and $V$ respectively.   It is easy to verify that we can find $\alpha, \beta, \gamma$ to satisfy the boundary conditions $u_1 (-\infty) = \bar{u}_1,
u_2 (\infty) = \bar{u}_2$ as well as the flux condition $J = k_s (\bar{u}_1 -\bar{u}_2)$.  We obtain (\ref{eq:sharp})$_2$.

%%%%%%%%%%%%%%%%%%%%%%%%%%%%%
\subsection{Homogenization of the sharp interface formulation}

Consider the situation where the domain $\Omega$ is made of a periodic microstructure at a scale $\varepsilon << 1$.  Specifically, let $Y$ be the unit cube consisting of two subdomains $Y_1$ and $Y_2$ separated by an interface $\Sigma$; $Y = Y_1 \cup Y_2\cup\Sigma$.  We assume that $\Omega^\varepsilon_1 = \cup_i \varepsilon (a_i + Y_1)$, $\Omega^\varepsilon_2 = \cup_i \varepsilon (a_i + Y_2)$, and ${\mathcal S}^\varepsilon = \cup_i \varepsilon (a_i + \Sigma)$.  We assume that the equations (\ref{eq:sharp}) hold in this domain with the reaction coefficient of order $\varepsilon$: i.e., $k_s = \varepsilon K_s$ for some $K_s>0$ independent of $\varepsilon$. Peter and B\"ohm \cite{peter2008different} (also see Auriault and Ene \cite{auriault}) show that this periodic system can be homogenized, and the homogenized equations are given by (\ref{eq:diffuse}) where $k_1, k_2$ are given by the usual unit cell problem of diffusion and 
$$
f_1 = - f_2 = K_s \mbox{Area} (\Sigma) (\bar{u}_1 - \bar{u}_2)
$$
where $\bar{u}_i$ is the solution to the unit cell problem.  Therefore, $k_i, f_i$ depend not only on the volume fraction but also other aspects of the microstructure.  However, we may view (\ref{eq:cond}) and (\ref{eq:source}) as simple models for these.

%%%%%%%%%%%%%%%%%%%%%%%%%%%%%
\subsection{Variational formulation}

%{\color{blue} 1.  The functional you had was incorrect.  The first term works only in 2d and the source term gave the wrong E-L equations.  I think all the proofs still work, but please double-check and change\\
%2.Incorporate an appropriate statement about the Euler-Lagrange equation into Theorem \ref{existence}}

The following theorem provides a variational formulation of the problem (\ref{eq:diffuse}) above.

\begin{theorem}
\label{existence}
Let $\Omega\subset\mathbb{R}^n$ be a bounded, connected open set with Lipschitz boundary,
$$\chi \in \mathcal{X} = \{\chi\in L^2(\Omega;[0,1])\}$$
be a given design and $\lambda \in {\mathbb R}$. The problem 
$$
 \inf\left\{L(u,\chi) = \int_\Omega \frac{1}{2} \sum_{i=1,2} k_i |\nabla u_i|^2 + \frac{1}{2} \chi(1-\chi)u \cdot A u -\lambda\chi\,dx\,:\, u\in\mathcal{V}\right\}
$$
where 
\begin{align*}
\mathcal{V}=\{v\in H^1(\Omega;\mathbb{R}^2):v_i=u_i^*\text{ on }\;\partial_{i}\Omega,\,i=1,2\}\\
\end{align*}
attains its minimum.  Further, the minimum is unique and satisfies the Euler-Lagrange equation
\begin{align} \label{eq:elw}
\int_\Omega \left( \sum_{i=1,2} k_i \nabla u_i \cdot \nabla \varphi_i + \chi(1-\chi) \varphi \cdot A u \right) dx
=0
\end{align}
for all $\varphi \in {\mathcal V}_0 = \{v\in H^1(\Omega;\mathbb{R}^2):v_i=0\text{ on }\;\partial_{i}\Omega,\,i=1,2\}$.
%
%Additionally, if the minimizing  $\bar{u}=\{\bar{u}_1,\bar{u}_2\}\in\mathcal{V}\cap C^2(\bar{\Omega})$, it also satisfies the following Euler-Lagrange equations
%\begin{align*}
%&\nabla\cdot k_i \nabla \bar{u}_i = f_i , \quad i = 1,2 \quad \text{in}\;\Omega\\
%&k_i \nabla \bar{u}_i\cdot \hat{n}=0\;\;\;\;\text{on}\;\partial\Omega\setminus\partial_{i}\Omega\\
%&\bar{u}_i=u_i^*\;\;\;\;\text{on}\;\partial_{i}\Omega.
%\end{align*}

\end{theorem}

\begin{proof}
Set 
$$
\inf\left\{L(u,\chi)\,:\, u\in\mathcal{V}\right\}=m
$$
and observe that because our integrand is finite and satisfies the growth conditions 
$$
-\lambda \chi(x)\leq f(x,v,\xi) \leq c(1+|v|^2+|\xi|^2),
$$
we have that $-\infty<m<+\infty$.  Let $u^\nu$ be a minimizing sequence, i.e. $L(u^\nu, \chi)\rightarrow m$ as $ \nu \to \infty$. For $\nu$ sufficiently large, 
$$
m+1\geq L(u^\nu,\chi)
\geq \gamma_1\norm{\nabla u^\nu}^2_{L^2}+\gamma_2\norm{u^\nu}^2_{L^2} -\int_\Omega |\gamma_3(x)|\,dx
\geq \gamma_1\norm{\nabla u^\nu}^2_{L^2}-\gamma_4
$$
with $\gamma_k>0$ independent of $\nu$ since $\Omega$ is bounded.   It follows that 
$$
\norm{u^\nu}_{W^{1,2}}\leq\gamma_5.
$$
appealing to our version of Poincar{\'e}'s inequality (Lemma \ref{poincare} below).
We deduce that there exists a $\bar{u}\in\mathcal{V}$  and a subsequence (still denoted $u^\nu$) that converges weakly in $W^{1,2}$:  $u^\nu\rightharpoonup \bar{u}$ in $W^{1,2}$ as $\nu\rightarrow \infty$.
It follows from the convexity of the integrand (since $k_1, k_2, k_s >0$) that the functional is sequentially weakly lower semicontinuous.  Therefore,
$$
\liminf_{\nu\rightarrow\infty}L(u^\nu,\chi)\geq L(\bar{u},\chi)
$$
and hence $\bar{u}$ is a minimizer of $(P)$. 

A simple calculation shows that any minimizer satisfies the Euler-Lagrange equation (\ref{eq:elw}).  We prove the uniqueness of the minimum by contradiction.  Suppose $L(u,\chi) = L(v,\chi) = m$.  Then, 
$$
\int_\Omega \left(  \frac{1}{2} \sum_{i=1,2} k_i (|\nabla u_i|^2 - |\nabla v_i|^2) + \frac{1}{2} \chi(1-\chi)
(u \cdot A u - v \cdot A v) \right) dx = 0.
$$
Further, since $u,v \in {\mathcal V}$, $u-v \in {\mathcal V}_0$.  Therefore, from the Euler-Lagrange equation (\ref{eq:elw}) for $v$, we conclude
$$
\int_\Omega \left( \sum_{i=1,2} k_i |\nabla v_i|^2  + \chi(1-\chi) v \cdot A v 
- \sum_{i=1,2} k_i \nabla u_i \cdot \nabla v_i - \chi(1-\chi) v \cdot A u \right) dx =0
$$
Adding these two equations, 
$$
 \frac{1}{2} \int_\Omega \left( \sum_{i=1,2} k_i |\nabla u_i - \nabla v_i|^2 + 
 \chi(1-\chi) (u-v) \cdot A (u-v) \right) dx =0.
 $$
It follows that $\nabla u_i = \nabla v_i$ a.e. and $u-v = \psi(x) \{1,1\}$.  Together, we conclude that $\psi$ is constant and from the boundary condition that $\psi=0$.  Thus $u=v$, giving us a contradiction.
\end{proof}

We have used the following lemma.

\begin{lemma}(Poincar{\'e}'s inequality, adapted from \cite{evans}) 
\label{poincare}
Let $\Omega$ and ${\mathcal V}$ be as in the theorem above.  There exists a constant $c$, depending only on $n$ and $\Omega$ such that 
$$
\norm{u}_{L^2} \leq c\norm{\nabla u}_{L^2}
$$
for each function $u\in \mathcal{V}$.
\end{lemma}
\begin{proof}
We argue by contradiction. Were the stated estimate false, there would exist for each positive integer $k$ a function $u^k\in \mathcal{V}$ satisfying 
$$
\norm{u^k}_{L^2} > k \norm{\nabla u^k}_{L^2}.
$$
We renormalize by defining 
$$
v^k:=\frac{\{u^k_1-u^*_1,u^k_2-u^*_2\} }{\norm{\{u^k_1-u^*_1,u^k_2-u^*_2\}}_{L^2}} 
$$
and note that $v^k \in {\mathcal V}_0$ and $||v^k||_{L^2} = 1$.  It follows that 
$$
\norm{\nabla v^k}_{L^2}<\frac{1}{k}.
$$
In particular the functions $\{v^k\}_{k=1}^\infty$ are bounded in $H^1$.  It follows (e.g., \cite{dacorognadirect}, Thm. 12.11) that there exists a subsequence $\{v^{k_j}\}_{k=1}^\infty\subset \{v^k\}_{k=1}^\infty$ and a function $v\in L^{2}(\Omega)$ such that 
$
v^{k_j}\rightarrow v\text{   in }L^2.
$
Further, the strong convergence implies that $v \in {\mathcal V}_0$ and $\norm{v}_{L^2}=1$.
On the other hand, the bound on $\nabla v^k$ from above implies that $\nabla v =0$ a.e., and that $v$ is constant since $\Omega$ is connected.  Since $v \in {\mathcal V}_0$, $v=0$ on $\Omega$ contradicting the conclusion $\norm{v}_{L^2} =1$.

\end{proof}

%
%%%%%%%%%%%%%%%%%%%%%%%%%%%%%%
%\subsection{Gamma convergence of diffuse interface problem to the sharp interface problem}
%
%{\color{blue} Perhaps a to do.  Leave for now.}

%%%%%%%%%%%%%%%%%%%%%%%%%%%%%%%%%%%%%%%%%%%%%%%%%
\section{Optimal design problem} \label{sec:opt}

We seek to find the arrangement of the two phases with a given volume of phase 1, $v$, that maximizes the normalized flux through the material:
\begin{equation}\label{eq:tp}
O := \sup \left\{
 \int_{\partial_1 \Omega} u_1^* k_1 \nabla u_1 \cdot \hat{n} \, dA - 
 \int_{\partial_2 \Omega} u_2^* k_2 \nabla u_2 \cdot \hat{n} \, dA
: \chi \in {\mathcal X}, \int_\Omega \chi dx = v
\right\} .
\end{equation}
Note that $k_i \nabla u_i \cdot \hat{n}$ gives the inward flux per unit area of species $i$ into $\Omega$.  We normalize each flux by the prescribed concentration.  Integrating by parts, using the variational characterization of the governing equations, and introducing a Lagrange multiplier to enforce the constraint on the given volume of phase 1, yields
\begin{equation}\tag{P} \label{eq:saddle}
O= \sup_{\chi\in \mathcal{X}}\inf_{u\in\mathcal{V}}\left\{L(u,\chi) = \int_\Omega \frac{1}{2} \sum_{i=1,2} k_i |\nabla u_i|^2 + \frac{1}{2} \chi(1-\chi)u \cdot A u -\lambda\chi\,dx\right\}.
\end{equation}

\section{Characterization of the optimal design problem} \label{sec:relaxed}

%%%%%%%%%%%%%%%%%%
\subsection{Saddle point theorem} \label{sec:saddle}
We begin by showing that we can exchange the order of finding the supremum and infimum in the saddle point problem (\ref{eq:saddle}).
\begin{theorem} \label{th:exchange}
There exists $\bar{v} \in {\mathcal V}$, $\bar{\chi} \in {\mathcal{X}}$ such that 
$$L(\bar{v},\bar{\chi})=\sup_{\chi\in \mathcal{X}}\inf_{v\in \mathcal{V}}L(v,\chi)=\inf_{v\in \mathcal{V}}\sup_{\chi\in \mathcal{X}}L(v,\chi).$$
for the saddle point problem (\ref{eq:saddle}).
\end{theorem}

The proof of this draws from the following theorem adapted from Ekeland and T\'emam \cite{et}.

\begin{theorem} [Proposition 2.4 of \cite{et}] 
\label{th:et}
Suppose two reflexive Banach spaces $V$ and $Z$ satisfy 
\begin{enumerate}[(i)]
\item $\mathcal{A}\subset V$ is convex, closed and non-empty, 
\item $\mathcal{B}\subset Z$ is convex, closed and non-empty. 
\end{enumerate}
Further let the function $L:\mathcal{A}\times\mathcal{B}\mapsto\mathbb{R}$ satisfy
 \begin{enumerate}[(i)]
  \setcounter{enumi}{2}
\item $\forall u\in\mathcal{A},\,p\rightarrow L(u,p)$ is concave and upper semicontinuous, 
\item $\forall p\in\mathcal{B},\,u\rightarrow L(u,p)$ is convex and lower semicontinuous,
\item there exists $p_0\in\mathcal{B}$ for $\mathcal{B}$ bounded such that 
$$
\lim_{\substack{u\in\mathcal{A}\\ \norm{u}\rightarrow\infty}}L(u,p_0)=+\infty.\\
$$
 \end{enumerate}
Then $L$ possesses at least one saddle point on $\mathcal{A}\times\mathcal{B}$.   
%Further, the set $\mathcal{A}_0\times\mathcal{B}_0$ of the saddle points of $L$ is convex.  Finally, if $p\rightarrow L(u,p)$ is strictly concave, $\forall u\in \mathcal{A},$ then $\mathcal{B}_0$ contains at most one point; if $u\rightarrow L(u,p)$ is strictly convex, $\forall p\in\mathcal{B}$, then $\mathcal{A}_0$ contains at most one point. 
\end{theorem}

We apply this theorem with $L$ as in problem (\ref{eq:saddle}), $V = W^{1,2}(\Omega; \mathbb{R}^2)$, $Z=L^2(\Omega;[0,1])$, ${\mathcal A} = {\mathcal V}$ and $\mathcal{B} = \mathcal{X}$.  Clearly, $V$ and $Z$ are reflexive Banach spaces as required by the theorem above.  We now show that these satisfy the rest of hypothesis (H).

\begin{proposition}
\label{spaceproperties}
Both $\mathcal{X}$ and $\mathcal{V}$ are convex, closed, and non-empty.
\end{proposition}
\begin{proof}
The point of concern is showing that our space $\mathcal{X}$ is in fact closed. So consider a sequence $\chi_i\in \mathcal{X}$ such that 
$
\chi_i {\rightharpoonup} \chi
$
in $L^2$.
We seek to show that the limit function $\chi\in \mathcal{X}$. From the definition of $\mathcal{X}$, 
$\norm{\chi_i}_{L^\infty}\leq1.$ Thus, we can pick a subsequence $\chi_{i_k}$ of $\chi_i$ such that
$
\chi_{i_k}\overset{\ast}{\rightharpoonup}\bar{\chi},
$
in $L^\infty$ as $k\to \infty$.  It follows
$
\chi_{i_k}{\rightharpoonup}\bar{\chi},
$
in $L^2$ as $k\to \infty$.  Therefore, $\chi = \bar{\chi}$ and $\esssup \chi \le 1$.  Similarly we can show
$\essinf \chi \ge 0$.  Thus, $\chi \in {\mathcal X}$ and $\mathcal{X}$ is closed. 
\end{proof}

\begin{proposition}
\label{lsc}
For each $\chi\in \mathcal{X},$ $v\mapsto L(v,\chi)$ is convex and lower semicontinuous.\end{proposition}
\begin{proof}
This follows trivially from the fact that the integrand in $L$ is a sum of a positive definite quadratic term in $\nabla u_i$ and a positive semidefinite quadratic form in $u$.

\end{proof}

\begin{proposition}
\label{usc}
For each $v\in \mathcal{V},$ $\chi\mapsto L(v,\chi)$ is concave and upper semicontinuous.
\end{proposition}
\begin{proof} 
This follows trivially from the fact that the integrand in $L$ is a sum of a positive definite quadratic term and two linear terms in $\chi$.
\end{proof}

\begin{proof}[Proof of Theorem \ref{th:exchange}]
From the aforementioned propositions, we have satisfied requirements $(i)-(iv)$ of the theorem.  To show $(v)$, set  $\chi(x)=1/2$.  We have
\begin{align*}
L(u,0.5)&= \int_\Omega 
\left( \frac{1}{2} \sum_{i=1,2} \bar{k}_i |\nabla u_i|^2 + \frac{1}{8}u \cdot A u -\frac{1}{2}\lambda \right) \,dx\\
&\geq \int_\Omega \left( \frac{1}{2} \sum_{i=1,2} \bar{k}_i |\nabla u_i|^2  -\frac{1}{2}\lambda \right) \, dx\\
&\geq  c_1 || \nabla u||_{L^2} -c_2
 \geq c_3 || u||_{H^1} -c_2
\end{align*}
for suitable positive constants $c_i$, where we use the derived form of Poincar{\'e}'s inequality in the final step.
The requirement $(v)$ follows.\\
\end{proof}

%%%%%%%%%%%%%%%%%%%%%%%%%%%%%
 \subsection{Explicit characterization}
 
 We are now ready to obtain the explicit characterization of the optimal design problem (\ref{eq:saddle}).
 
 \begin{theorem}
 We have
 $$
 O = \min_{v\in\mathcal{V}} \int_\Omega \overline{W}(v,\nabla v)\,dx.
 $$
 where 
% \[
%\bar{W}(v,\xi)=\max\begin{cases}
%\frac{\Psi(v,\xi)}{8 (v_1-v_2)^2}\left(|\nabla v_1|^4 \Delta k_1^2+|\nabla v_2|^4 \Delta k_2^2+2 |\nabla v_1|^2|\nabla v_2|^2 \Delta k_1 \Delta k_2\right.\\
%\;\;\;\;+2 |\nabla v_1|^2 \left(-2 \lambda  \Delta k_1+(k_{11}+k_{12}) (v_1-v_2)^2\right)\\
%\;\;\;\;\left.+2 |\nabla v_2|^2 \left(-2 \lambda  \Delta k_2+(k_{21}+k_{22}) (v_1-v_2)^2\right)\right.\\
%\;\;\;\;+\left.\left((v_1-v_2)^2-2 \lambda \right)^2\right),\\
%\frac{1}{2} \left(|\nabla v_1|^2 k_{11}+|\nabla v_2|^2 k_{21}-2 \lambda \right),\\
%\frac{1}{2} \left(|\nabla v_1|^2 k_{12}+|\nabla v_2|^2 k_{22} \right)
%\end{cases}
%\]
%
$$
\overline{W}(v,\xi) = \begin{cases}
\frac{1}{2} \left(|\xi_1|^2 k_{21}+|\xi_2|^2 k_{22} \right)
& (v,\xi) \in {\mathcal R}_0,
\\ \\
\displaystyle{
\frac{ \left(\sum_i \Delta k_i |\xi_i|^2\right)^2 
+2 \sum_i |\xi_i|^2 (k_v(k_{i1}+k_{i2}) -  2\lambda \Delta k_i)
+ (k_v - 2 \lambda)^2}{8k_v}}  & (v,\xi) \in {\mathcal R},
\\ \\
\frac{1}{2} \left(|\xi_1|^2 k_{11}+|\xi_2|^2 k_{12}-2 \lambda \right) 
& (v,\xi) \in {\mathcal R}_1
\end{cases} 
$$
with
\begin{align*}
{\mathcal R}_0 &= \{ (v,\xi): \sum_i \Delta k_i |\xi_i|^2 - 2\lambda \leq -k_v \},  \\
{\mathcal R} &= \{ (v,\xi): -k_v < \sum_i \Delta k_i |\xi_i|^2 - 2\lambda < k_v \}, \\
{\mathcal R}_1 &= \{ (v,\xi): \sum_i \Delta k_i |\xi_i|^2 - 2\lambda \geq k_v \},
\end{align*}
and 
\begin{align*}
\Delta k_i = k_{i1} - k_{i2} , \quad
k_v = k_s(v_1-v_2)^2.
\end{align*}

\end{theorem}

%The function $\bar{W}$ is shown in Figure \ref{fig:Wbar}. We use $\Psi(v,\xi)$ as an indicator function for the shaded region $\mathcal{R}$, seen in Figure \ref{fig:Wbar}(a), where intermediate densities are favorable for calculated concentration fields and their associated gradients. The plot seen in Figure \ref{fig:Wbar}(b) is a plot of the energy as a function of $(\nabla v)$, for a sample set of transport parameters, and for a fixed value of $v=\{v_1,v_2\}$. The opposing sides to the well correspond to the discrete phases $\chi=\{0,1\}$, and region $\mathcal{R}$ acts as a smooth transition between these two. We use that fact that maximizing over convex functions yields a convex function. In other words, when we maximize $W(v,\xi,\chi),$ which is convex in $|\nabla v|$, over $\chi$, we must obtain a convex function in $|\nabla v|$ regardless of the value that $\chi$ takes.\\
The function $\overline{W}$ is shown in Figure \ref{fig:Wbar} as a function of $\xi$ for a fixed $v$ with various parameters. The shaded regions indicate the gradients for which mixed phases ($\chi\in(0,1)$) occur. Note that mixed phases occur where the gradients of both species are comparable in magnitude, and pure phases occur otherwise. 

\begin{figure}
\label{fig:Wbar}
\centering
\vspace{0.5in}
	\includegraphics[width=0.7\textwidth]{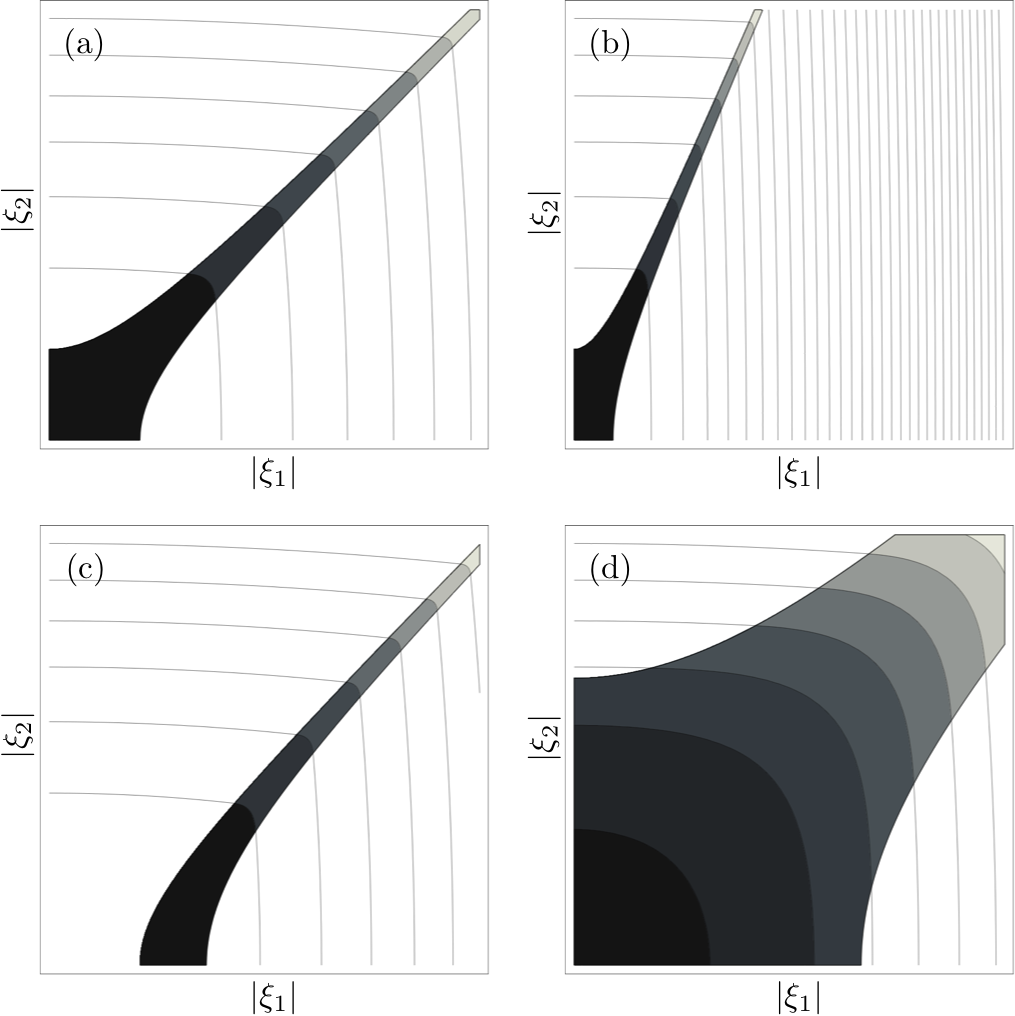}
\vspace{0.5in}
\caption{Contour plot of $\overline{W}$ for fixed $v,\lambda$.  The shaded regions indicate where mixed phase $(\chi\in(0,1))$ occurs.  (a) $k_{11} = k_{22} = k_s = 1, k_{12} = k_{21} =0.1, (v_1 - v_2)^2=1, \lambda = 0$.
(b) Parameters as in (a) except $k_{11} = 5$, (c) Parameters as in (a) except $\lambda = 1$, (d) Parameters as in (a) except $(v_1-v_2)^2 = 10 $.  }
\end{figure}

\begin{proof}
For $v \in {\mathbb R}^2, \xi \in {\mathbb R}^{2\times2}, \chi \in {\mathbb R}$, set 
$$
W(v,\xi,\chi) =  \frac{1}{2}
\sum_{i=1,2} (\chi k_{i1} + (1-\chi)k_{i2}) |\xi_i|^2 
+ \frac{k_s}{2} \chi(1-\chi) v \cdot A v -\lambda\chi  
$$
and
\begin{align} \label{eq:barw}
\overline{W}(v,\xi) = \max_{\chi \in [0,1]} W(v,\xi,\chi) .
\end{align}
In light of the saddle point theorem,
$$
O = \inf_{v\in \mathcal{V}}\sup_{\chi\in \mathcal{X}} \int_\Omega W(v, \nabla v, \chi) dx = 
\inf_{v\in \mathcal{V}} \int_\Omega \overline{W} (v,\nabla v) dx.
$$
It remains to compute $\overline{W}$.   To that end, note that for a fixed $v,\xi$, $W(v,\xi,\chi)$ is quadratic in $\chi$ and 
$$
\frac{\partial W}{\partial \chi}(v,\xi,\chi) = 0
$$
has an unique solution $\chi = \chi^*$.  A simple calculation reveals
$$
\chi^*(v,\xi)=\frac{ \sum_i \Delta k_i |\xi_i|^2  + k_s(v_1-v_2)^2 -2 \lambda }{2 k_s (v_1-v_2)^2}.
$$
Notice that 
$$
\chi^*(v,\xi) \begin{cases}
\le 0 & (v,\xi) \in {\mathcal R}_0, \\
\in (0,1) &  (v,\xi) \in {\mathcal R},\\
\ge 1 &  (v,\xi) \in {\mathcal R}_1.
\end{cases}
$$
A longer, but straightforward, calculation reveals that 
$$
W(v,\xi,\chi^*) = \frac{ \left(\sum_i \Delta k_i |\xi_i|^2\right)^2 
+2 \sum_i (|\xi_i|^2 k_v(k_{i1}+k_{i2}) -  2\lambda \Delta k_i)
+ (k_v - 2 \lambda)^2}{8k_v} .
$$ 
Similarly, 
\begin{align*}
W(v,\xi,0) &= \frac{1}{2} \left(|\xi_1|^2 k_{21}+|\xi_2|^2 k_{22} \right), \\
W(v,\xi,1) &= \frac{1}{2} \left(|\xi_1|^2 k_{11}+|\xi_2|^2 k_{12}-2 \lambda \right) .
\end{align*}

Now, we can verify by explicit calculation that 
\begin{align} \label{eq:wcomp}
W(v,\xi, \chi^*) - W(v,\xi, 0) &= \frac{k_v}{2} (\chi^*(v,\xi))^2  \nonumber \\
W(v,\xi, \chi^*) - W(v,\xi, 1) &= \frac{1}{4} (\chi^*(v,\xi)-1 )^2 \\
W(v,\xi, 1) - W(v,\xi, 0) &= \frac{1}{2} \left( \sum_i \Delta k_i |\xi_i|^2 - 2 \lambda \right)  \nonumber 
\end{align}

We obtain the desired result by recalling (\ref{eq:barw}), rewriting
$$
\overline{W} (v,\xi) = \max \{ \Psi(v,\xi) W(v,\xi, \chi^*), W(v,\xi, 0), W(v,\xi,1) \}
$$
where 
$$
\Psi(v,\xi) = \begin{cases}
1 & (v,\xi) \in {\mathcal R}\\
- \infty & else
\end{cases}
$$
and using (\ref{eq:wcomp}).

\end{proof}

%%%%%%%%%%%%%%%%%%%%%%%%%
\section{Phase-field formulation of the optimal design problem} \label{sec:phasefield}

The min-max problem based on the functional $L$ is difficult to solve numerically due to the fact that $\chi$ is only in $L^2$ and because of the constraint $\chi \in [0,1]$. The relaxed functional is also difficult to solve numerically since $\overline W$ is not strictly convex. Therefore, we now pursue an alternative approach to the optimal design problem that is amenable to numerical treatment.  We regularize the functional $L(u,\chi)$ by adding the $L^2$ norm of $ \nabla \chi$ and requiring $\chi \in H^1$. We also replace the constraint $\chi \in [0,1]$ with a penalty. Finally, from a practical point of view, it would also be beneficial to have solutions that prefer the pure phases $\chi \in \{0,1\}$. Therefore, we add a term to the energy that penalizes any deviation from this set.

We consider the functional
$$
\mathcal{L}(u,\chi) = \int_\Omega \left( \frac{1}{2} \sum_{i=1,2} k_i |\nabla u_i|^2 + \frac{1}{2} \chi(1-\chi)u \cdot A u -\lambda\chi - \left( \alpha W(\chi) +\beta |\nabla\chi|^2  \right) \right) \,dx,
$$
where 
$$
W(\chi)= \chi^2(1-\chi)^2,
$$
has two wells at $\chi \in \{0,1\}$.  

The additional terms in parenthesis form the integrand of the Allen-Cahn functional \cite{allencahn}. Minimizers of this functional partition the domain into regions where $\chi \approx 0$ and $\chi \approx 1$ separated by transition layers with thickness $\sim\sqrt{\beta/\alpha}.$ In our setting, we expect this to be modified by the transport energy.
%
%{\color{blue} Verify and state as theorem. We can follow the methods of Section \ref{sec:saddle} to prove that this functional has a saddle point $(u^*,\chi^*)$
%$$
%\min_{u\in {\mathcal V}} \ \max_{\chi \in H^1}  {\mathcal L}(u,\chi)
%=  \max_{\chi \in H^1} \ \min_{u\in {\mathcal V}}  {\mathcal L}(u,\chi)
%=  {\mathcal L}(u^*,\chi^*)
%$$
%}

We seek to find the saddle point by considering a gradient flow:
\begin{align*}
\int_\Omega \frac{\partial \chi}{\partial t}\varphi \ dx 
&= \frac{1}{d_\chi} \left\langle \delta_\chi, \varphi \right\rangle \\
\int_\Omega \frac{\partial u}{\partial t}\psi \ dx 
&= - \frac{1}{d_u} \left\langle \delta_u, \psi \right\rangle
\end{align*}
for every $\varphi, \psi \in H^1(\Omega;\mathbb{R}^N)$ subject to the appropriate boundary conditions where $\langle \cdot \rangle$ denotes the $L^2$ inner product and $d_\chi,d_u > 0$ are the inverse mobilities.
We obtain the following system of equations:
\begin{align}
&d_\chi \frac{\partial \chi}{\partial t}=\sum_{i=1,2}\frac{k_i'}{2}|\nabla u_i|^2+\frac{1}{2}u\cdot Au(1-2\chi)-\lambda  + \beta\nabla^2\chi-\alpha W'(\chi), \label{eq:numchi}\\
&d_u\frac{\partial u_i}{\partial t}=\nabla\cdot k_i \nabla u_i-\chi(1-\chi) A_{ij} u_j . \label{eq:numu}
\end{align}

%%%%%%%%%%%%%%%%%%%%%%%%%
\section{Numerical study of the optimal design problem}

We have implemented the phase field formulation of the optimal design problem (\ref{eq:numchi}, \ref{eq:numu}) using the commercial software COMSOL \cite{comsol}. All our simulations are in two dimensions ($n=2$).  We work with non-dimensional units where the size of the domain, the concentration at a boundary and the (diagonal components of the) diffusion coefficient are $\mathcal{O}(1)$. We discretize the problem spatially using linear finite elements generated by Delaunay triangulation, and integrate the resulting ordinary differential equation in time by using the backward differentiation formula. We impose the volume constraint as a global constraint that is built into COMSOL.  Additionally, we impose a point-wise constraint restricting $\chi \in [0,1]$. We typically begin with an initial guess of uniform $\chi$, and run the simulations until an apparent steady state is reached (i.e., when the right hand sides of (\ref{eq:numchi}, \ref{eq:numu}) become small compared to a given tolerance). The simulations can get stuck in local optima, but we try to avoid this by doing parameter sweeps and studying additional initial conditions.
 
%%%%%%%%%%%%%%%%%%%%%%%%%
\subsection{Square reactor}

We begin with a square domain, $\Omega = (0,1)^2$, shown in Figure \ref{fig:square1}(a).  We prescribe $u_1 = 1$ on the left face $\partial_1 \Omega = \{0\} \times (0,1)$ corresponding to a source of species 1, $u_2 = 0$ on the right face $\partial_2 \Omega = \{1\} \times (0,1)$ corresponding to a sink of species 2, and zero-flux boundary conditions otherwise.  We also impose a zero flux boundary condition on our phase-field variable $\chi$.

%%%%%%%%%%%  Simple example
\begin{figure}
\centering     
    \includegraphics[width=0.3\textwidth]{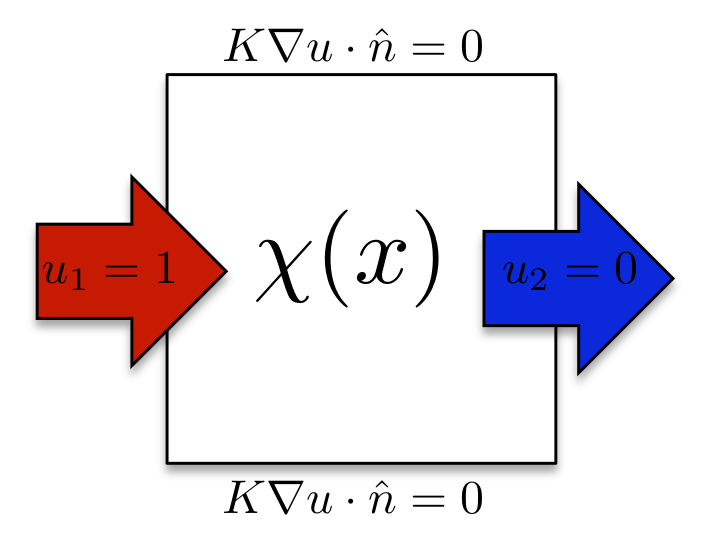}
    \includegraphics[width=0.6\textwidth]{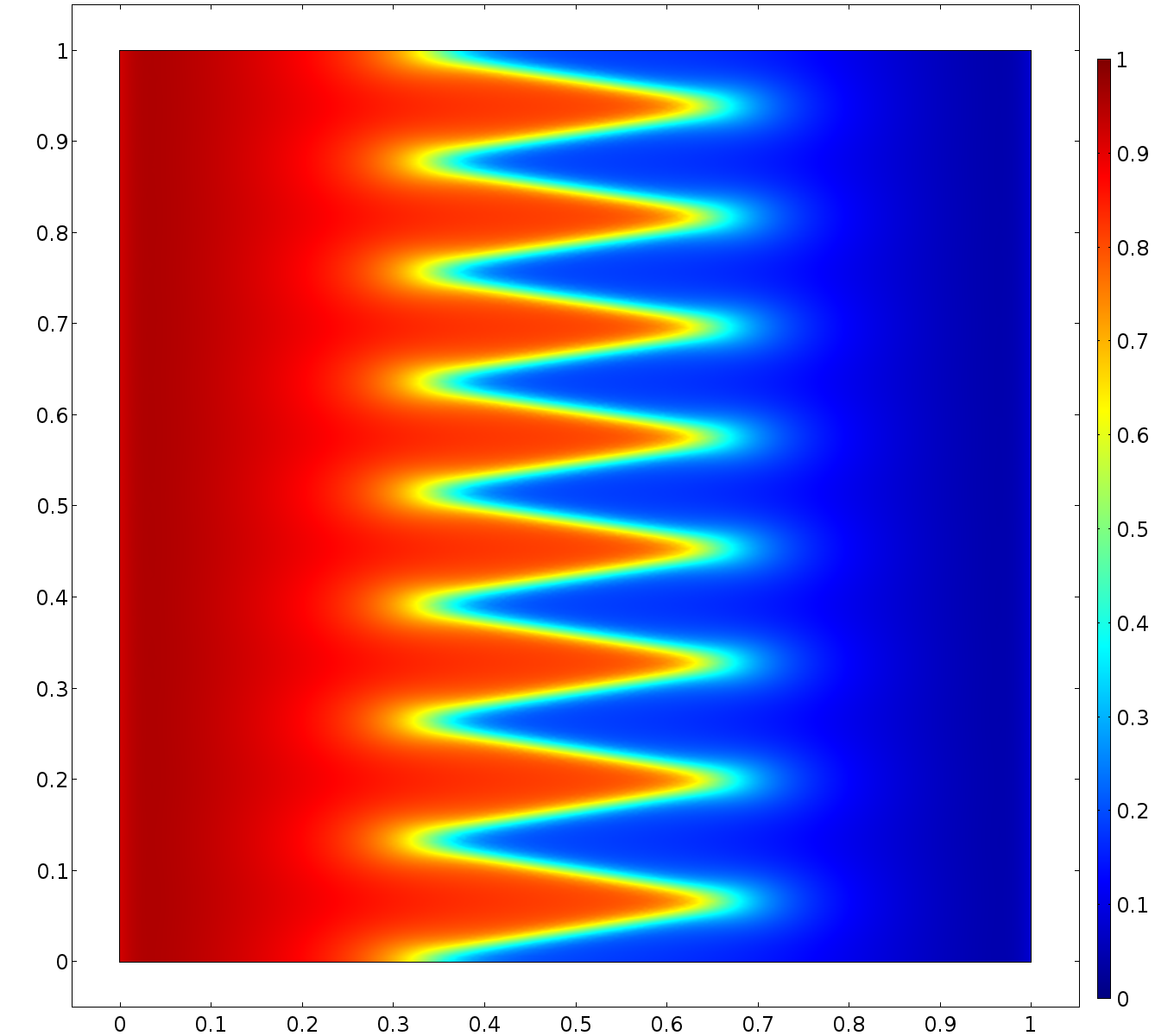} \\
    \hspace{-1in}(a)     \hspace{2.7in}(b)
    \caption{(a) Square reactor with a source of species 1 on the left and a sink of species 2 on the right. (b) Optimal design ($\chi$) for the parameters in (\ref{eq:param1})). \label{fig:square1}}
\end{figure}

The resulting optimal design $\chi$ is shown in Figure \ref{fig:square1}(b) for the parameters
\begin{gather}
k_{11} = k_{22} = 1, k_{12} = k_{21}= 1\times10^{-6}, k_s=1\times10^2, \nonumber \\
\alpha = 1,  \beta =  2\times10^{-5}, d_\chi = 2\times10^{-2}, d_u =2\times10^{-3}, \label{eq:param1} \\
v = 0.5. \nonumber
\end{gather}
This simulation had a mesh with 67068 elements, took $50$ non-dimensional units of time over 845 time steps and the $L^2$ norm of the time derivative of $\chi$ is $3.945\times10^{-4}$ at the end of the simulation. We have verified that the design does not change by refining the mesh and driving the $L^2$ norm of the time derivative of $\chi$ to $10^{-12}$.

The resulting design has a clear intuitive explanation.  Given the boundary conditions, the design seeks to draw in species 1 from the left, react it in the center to convert species 1 to species 2, and expel species 2 at the right.   Therefore, the design puts material 1, which has a high diffusivity of species 1, on the left so that it can easily transport species 1 from the source to the interface where the reaction consumes it. Material 2, which has a high diffusivity of species 2, is placed on the right so that it can easily transport species 2 from the interface, where the reaction generates it, to the sink. The design maximizes the reaction by creating a zig-zag interface between the two materials.

%%%%%%%%%%%  Volume fraction 0.5
\begin{figure}
\centering     
    \includegraphics[width=0.7\textwidth]{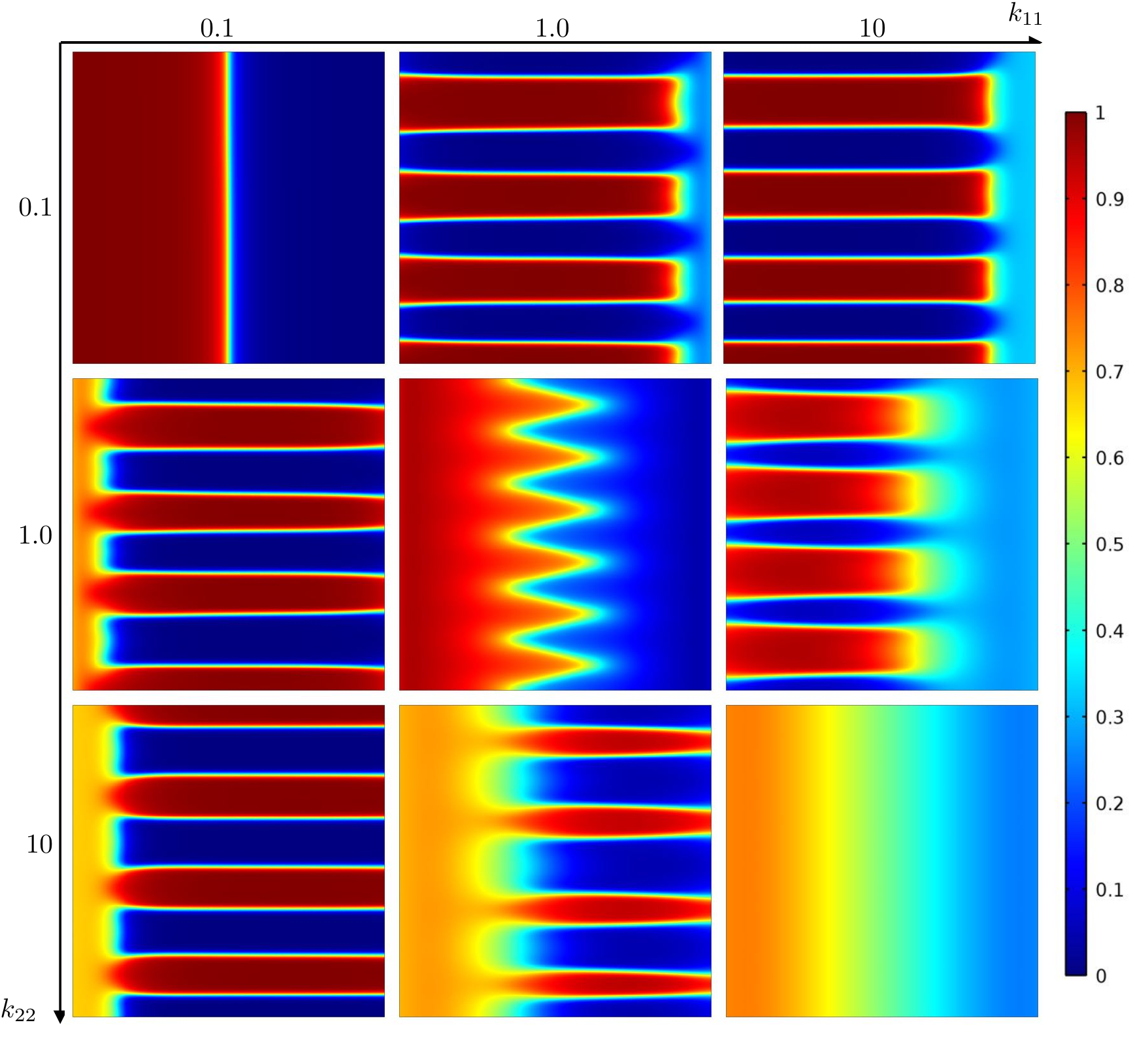}
    \caption{Designs with volume fraction $v=0.5$ as we vary diffusion coefficients with $\alpha=0.1,\,\beta=5\times10^{-5},\,k_{12}=10^{-3}\times k_{11},\,k_{21}=10^{-3}\times k_{22},\, d_\chi=1\times10^{-2}-1.5\times10^{-2},\,d_u=7\times10^{-4}-1\times10^{-3},\,k_s=1\times10^2$.  \label{fig:v05} 
}
\end{figure}

\begin{figure}
\centering     
    \includegraphics[width=0.7\textwidth]{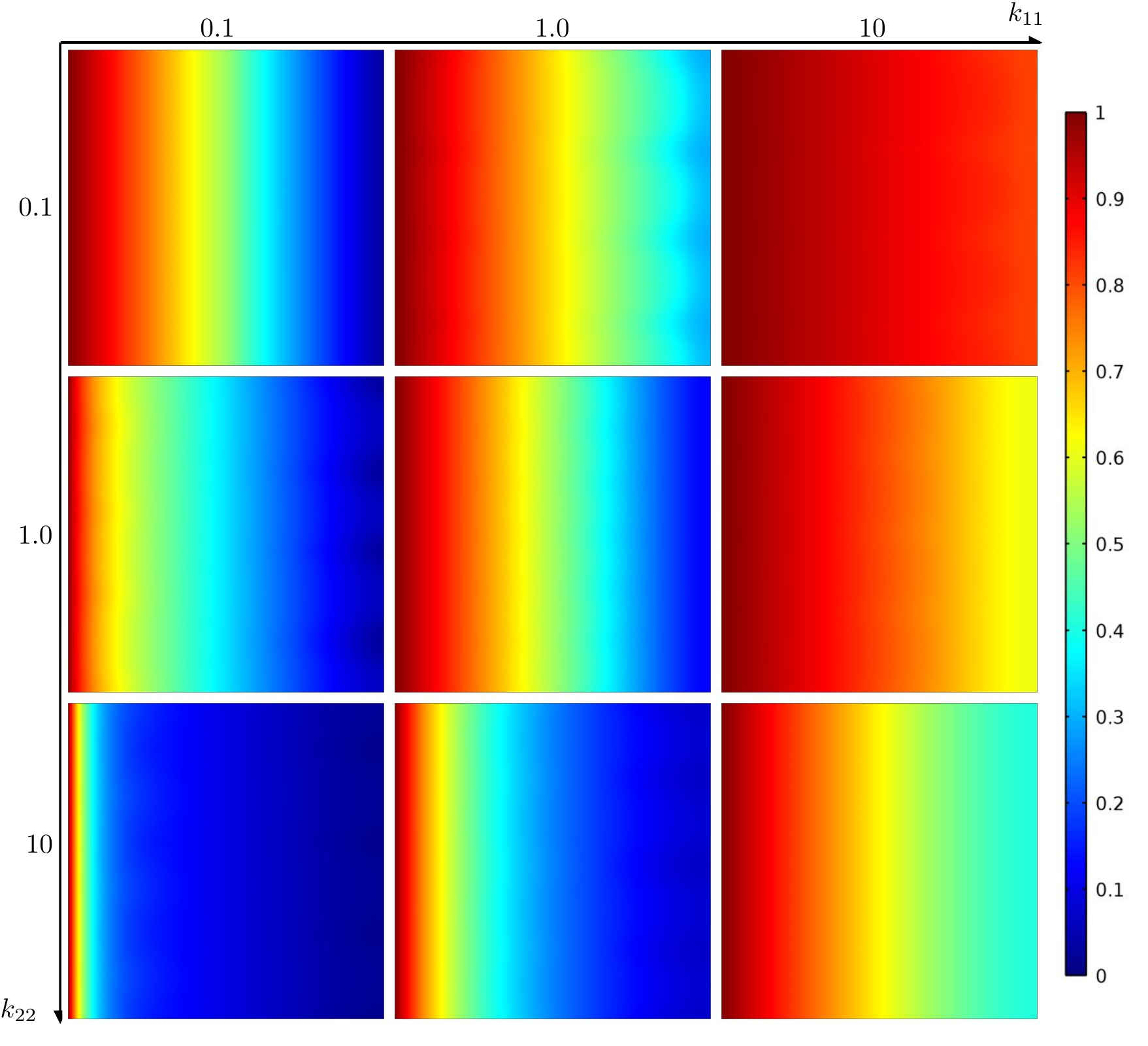}
    \caption{Concentration field $u_1$ associated with the designs presented in Figure \ref{fig:v05}.
    \label{fig:v05_u1} }
    \end{figure}

\begin{figure}
\centering     
    \includegraphics[width=0.7\textwidth]{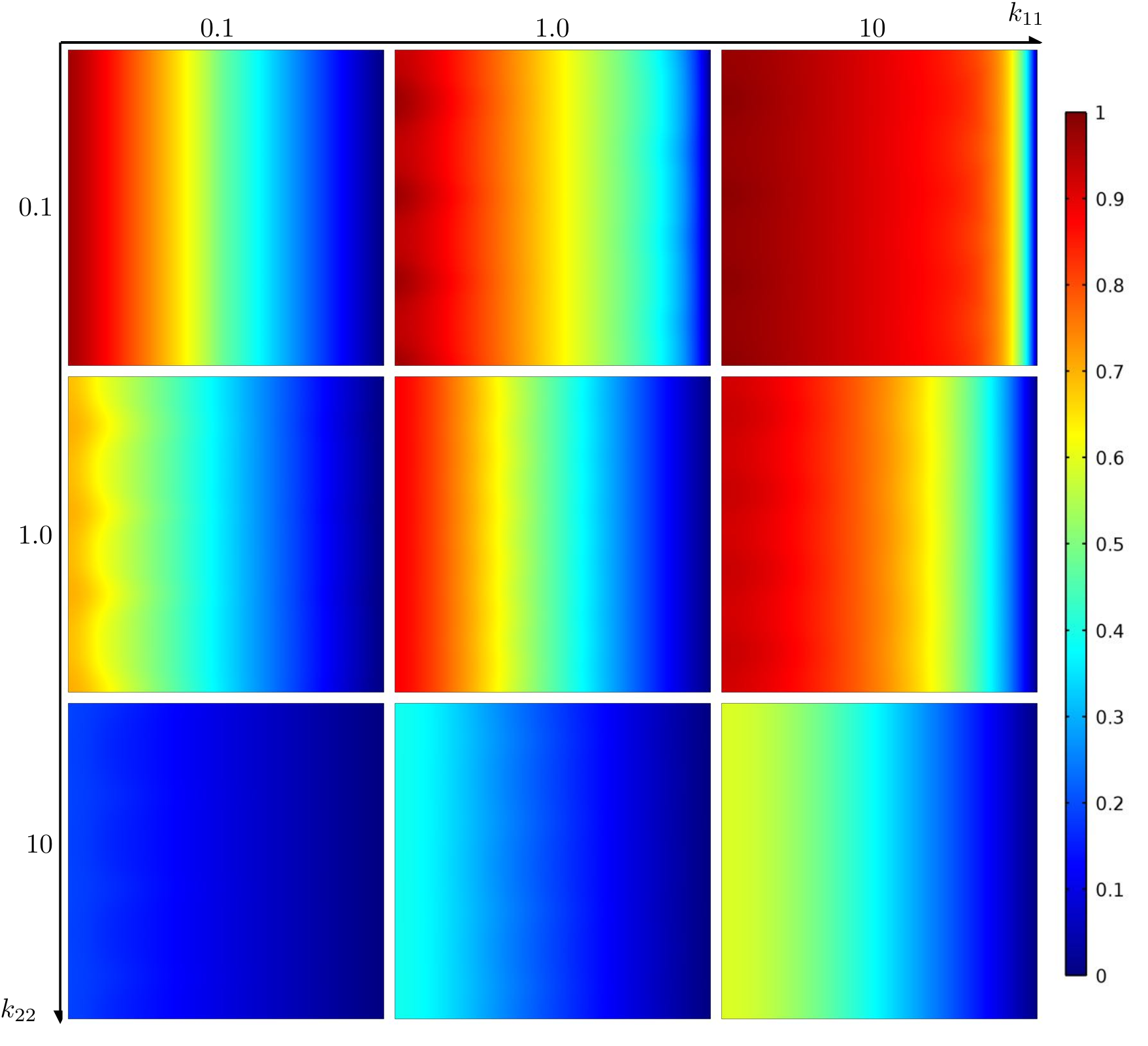}
    \caption{Concentration field $u_2$ associated with the designs presented in Figure \ref{fig:v05}.
    \label{fig:v05_u2} }
\end{figure}

\begin{figure}
\centering     
    \includegraphics[width=0.7\textwidth]{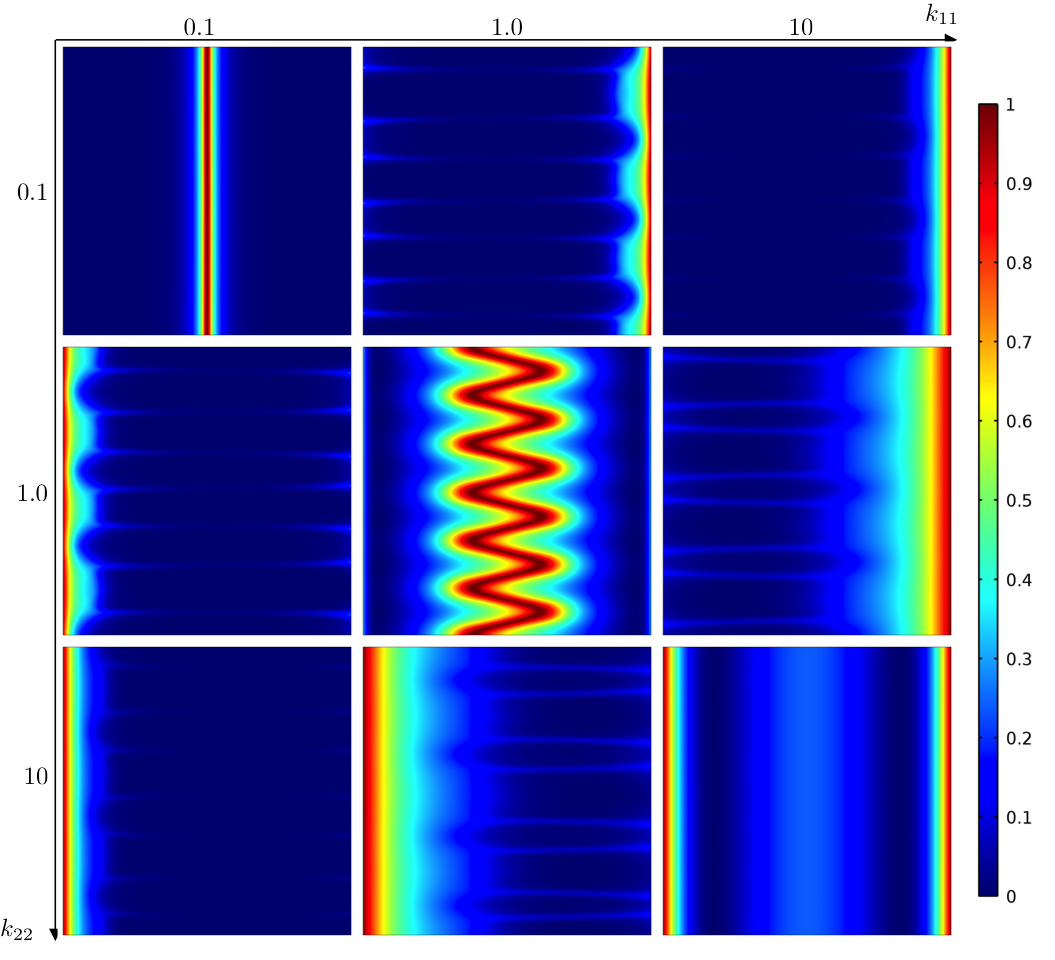}
    \caption{Distribution of reaction zones associated with the designs presented in Figure \ref{fig:v05}; normalized units.
    \label{fig:v05_rxn} 
}
\end{figure}

We now begin a parameter study for the same problem. Figure \ref{fig:v05} shows the resulting designs for a volume fraction $v = 0.5$ for various diffusivities $k_{11}, k_{22}$. Figures \ref{fig:v05_u1} and \ref{fig:v05_u2} show the corresponding concentration fields $u_1$ and $u_2$ respectively while Figure \ref{fig:v05_rxn} shows the corresponding reactions.

We begin at the center for the case $k_{11} = k_{22} = 1$, which is what we described earlier.  Decreasing both diffusivities by moving up on the diagonal to $k_{11} = k_{22}=0.1$ leads to a similar segregation of the material but the interface is sharper and straight.  On the other hand, increasing both diffusivities by moving down the diagonal to $k_{11} = k_{22}=10$ still segregates the material, but in a very diffuse manner with an almost constant gradient.  Note that the interface width  changes despite the fact that length-scale, $\sqrt{\beta/\alpha}$, predicted by the phase-field alone is held fixed.  This is because of the relative importance of the diffusion and the reaction.  When the diffusivities are both small, $k_{11} = k_{22}=0.1$ as in the upper-left, the reaction is relatively easy and diffusion difficult.  Thus one only needs a narrow region for the reaction, saving much of the pure material for optimal transport.  Conversely, when the diffusivities are both large, $k_{11} = k_{22}=10$ as in the bottom-right, the reaction is relatively difficult and diffusion easy.  Thus, one creates a very diffuse interface to optimize the reaction.

We now turn to the situation when the diffusivities are different.  Consider the case when $k_{11} = 1, k_{22} = 0.1$ as shown on the top-center.  The diffusion of species 1 is considerably easier than that of species 2.  Therefore, it is advantageous to have the reaction close to the sink.  Species 1 is transported by the long arms of material 1 (red) which protrude from the left to the right where it reacts very close to the sink, thereby reducing the distance that species 2 has to be transported.  The excess material 2 (blue) is `hidden' on the left in arms that do not participate in the transport.  The case $k_{11} = 10, k_{22} = 0.1$ shown on the top-right is similar with a slightly wider interface since reaction is more difficult compared to the transport. The case $k_{11} = 10, k_{22} = 1$ shown on the right-middle is also similar except the interfacial region is even wider.  The cases $k_{11} = 0.1, k_{22} = 1$; $k_{11} = 0.1, k_{22} = 10$ and $k_{11} = 1, k_{22} = 10$ are the analogous, with the roles of material 1 and 2 reversed.

%{\color{blue} It would be nice to plot the corresponding u1 and u2 fields?}
%The corresponding concentration fields for $u_1$ and $u_2$ are shown in \ref{fig:v05_u1} and \ref{fig:v05_u2} respectively. A local representation of the chemical reaction occurring is also presented in \ref{fig:v05_rxn}. These solutions obtained, and the values presented in \ref{tab:v05} are obtained by performing a separate, time-independent, study using the design result from the optimization procedure to eliminate any iterative errors. 

%\begin{figure}
%\centering     
%    \includegraphics[width=0.315\textwidth]{diff1_chi.png}
%    \includegraphics[width=0.30\textwidth]{diff1_u2.png}
%    \includegraphics[width=0.30\textwidth]{diff1_u2.png}
%    \includegraphics[width=0.025\textwidth]{diff1_scale.png}\\
%    (a) \hspace{1.8in} (b) \hspace{1.8in} (c)\\
%    \caption{(a) $\delta\chi=\chi_1(x,y)+\chi_2(1-x,1-y)-1$ with $\norm{\delta\chi}=0.106063$, (b) difference in $u_1$ field, $\norm{\delta u_1}=9.8890\times10^{-4}$, and (c) difference in $u_2$ field, $\norm{\delta u_2}=9.7364\times10^{-4}$. {\color{blue} What norm is this?} \label{fig:v05_diff}}
% \end{figure}

The phase-field functional, the domain, and the boundary conditions have a symmetry, and we examine if the resulting designs reflect this symmetry.  Specifically, note that if $\{u_1, u_2, \chi\}$ is a solution for a problem with $k_1,k_2$ on the square domain, then $\{1-u_2, 1-u_1, 1-\chi\}$ is a solution for a problem with $k_2,k_1$ on the square domain obtained by changing $x$ to $1-x$.  We see that our designs reflect this symmetry.  Specifically, compare the case $k_{11}=0.1,k_{22}=1$ and the resulting design $\chi_1$ shown in middle-left of Figure \ref{fig:v05} and the case $k_{11}=1,k_{22}=0.1$ and the resulting design $\chi_2$ shown in top-center of Figure \ref{fig:v05}.  We see that $\chi_1(x,y) \approx 1-\chi_2(1-x,1-y)$.  
%Indeed, Figure \ref{fig:v05_diff} shows that the difference $\delta\chi=\chi_1(x,y) - (1-\chi_2(1-x,1-y))$ in the design, and the associated concentration are extremely small.

\begin{table}
\centering
\caption{Contributions to the energy functional $L(u,\chi)$, the phase field regularization, the flux $J_i$ of each species calculated at the boundary, and the reaction (right hand side of Eq. (\ref{eq:diffuse})$_1$) integrated on the domain for $v=0.5$. \label{tab:v05}}
\vspace{2mm}
\resizebox{\columnwidth}{!}{%
\begin{tabular}{l*{8}{c}r}
\hline\hline
 $k_{11}$ & $k_{22}$ & Transport Energy & Reaction Energy & Phase Field & $J_{1,\text{in}}$ & $J_{2,\text{out}}$ & Reaction \\
  &  & $\int\frac{1}{2}\sum k_i|\nabla u_i|^2$ & $\int\frac{1}{2}\chi(1-\chi)u\cdot Au$ &  & $\int k_1\nabla u_1\cdot\hat{n}$ &$-\int k_2\nabla u_2\cdot\hat{n}$ & \\
 \hline
0.1	&	0.1	&	0.0451	&	0.0023	&	0.0203	&	0.0948	&	0.0948	&	0.0949	\\
0.1	&	1	&	0.1706	&	0.0285	&	0.0748	&	0.3866	&	0.3977	&	0.3983	\\
0.1	&	10	&	0.2952	&	0.2022	&	0.0779	&	0.9610	&	0.9947	&	0.9948	\\
1	&	0.1	&	0.1706	&	0.0285	&	0.0550	&	0.3977	&	0.3864	&	0.3983	\\
1	&	1	&	0.4276	&	0.0340	&	0.1070	&	0.9202	&	0.9201	&	0.9232	\\
1	&	10	&	0.9044	&	0.3595	&	0.1502	&	2.5015	&	2.5256	&	2.5278	\\
10	&	0.1	&	0.2953	&	0.2021	&	0.0574	&	0.9946	&	0.9602	&	0.9947	\\
10	&	1	&	0.9044	&	0.3596	&	0.1150	&	2.5257	&	2.5011	&	2.5278	\\
10	&	10	&	2.2730	&	0.9990	&	0.1699	&	6.5257	&	6.5254	&	6.5440	
\end{tabular}
}
\end{table}
Table \ref{tab:v05} shows how the different contributions to the energy change for the various cases. It also shows how the flux varies. Further, it shows the the flux at the source, sink, and reaction zone all agree.
%%%%%%%%%%%  Volume fraction 0.3
\begin{figure}
\centering     
    \includegraphics[width=0.7\textwidth]{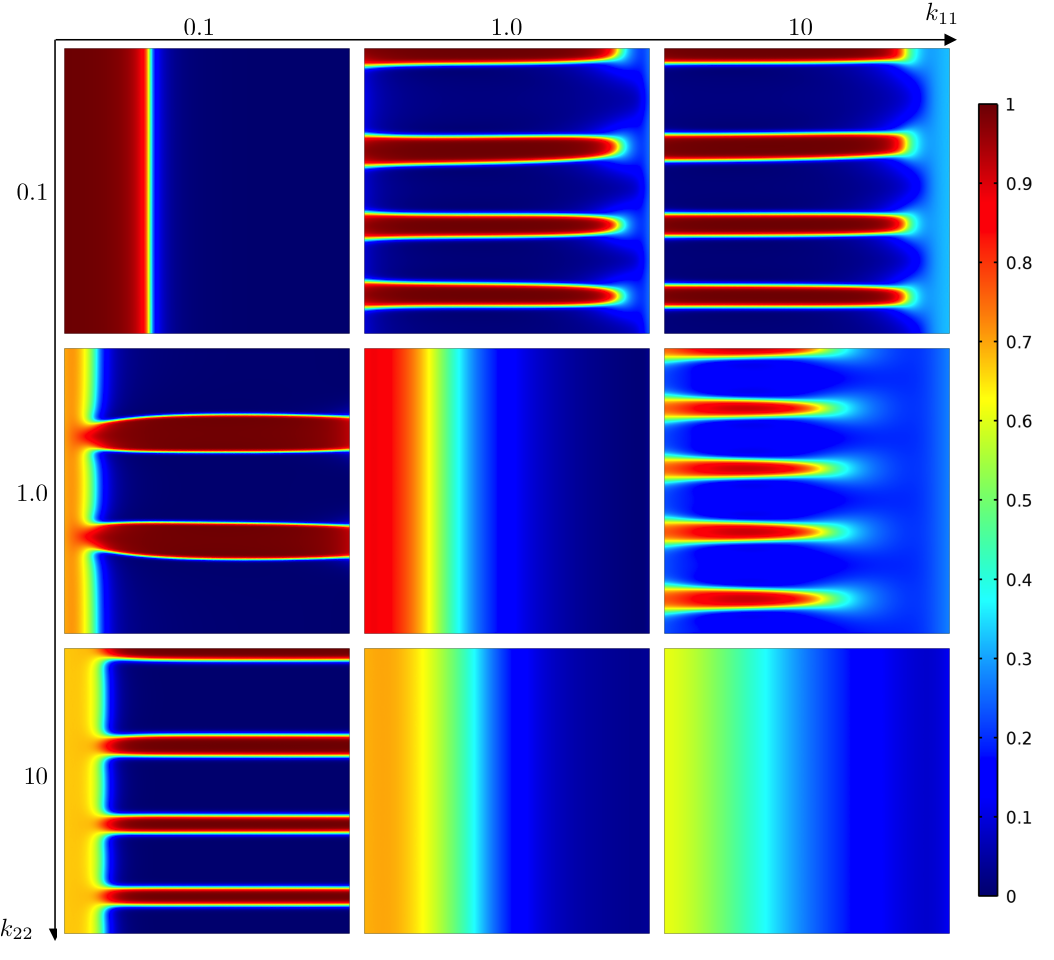}
    \caption{Parameter sweep with $v=0.3,\,\alpha=0.1,\,\beta=5\times10^{-5},\,k_{12}=10^{-3}\times k_{11},\,k_{21}=10^{-3}\times k_{22},\, d_\chi=1\times10^{-2}-2\times10^{-1},\,d_u=7\times10^{-4}-2\times10^{-2},\,k_s=1\times10^2$.\label{fig:v03}}
\end{figure}

%\begin{table}
%\centering
%\caption{{\color{blue} $v=0.3$ What does this table convey?   Split the transport energy to the transport energy and the reaction energy?  Should we report the overall flux?}}
%\begin{tabular}{l*{4}{c}r}
%\hline\hline
%$k_{11}$ & $k_{22}$ & Transport Energy & Phase Field  \\
%\hline
%0.1&	0.1&	0.04743&	0.01960\\
%0.1&	1&	0.37641	&0.04574\\
%0.1&	10	&0.52339&	0.07112	\\
%1&	0.1&	0.14867&	0.04618\\
%1&	1&	0.45681&	0.08023\\
%1&	10&	1.36629&	0.09704\\
%10&	0.1&	0.44340&	0.05073\\
%10&	1&	1.07605&	0.07702\\
%10&	10&	3.03093&	0.09210
%\end{tabular}
%\end{table}

Figure \ref{fig:v03} shows the designs for the same parameters, but for a volume fraction $v=0.3$.  The designs are similar, except the interface is more to the left.

%%%%%%%%%%%%%%%%%%%%%%%%%%%%%%%%%%%%%%
\subsection{Cylindrical reactor}
\begin{figure}
\label{fig:cyl}
\centering 
    {\small \hspace{0.2in} $\alpha=0.1,\beta=2\times 10^{-6}$ \hspace{0.4in} 
    $\alpha=1,\beta=2\times 10^{-5}$ \hspace{0.4in} 
    $\alpha=10, \beta=2\times 10^{-4}$}\\
    \rotatebox{90}{\small \phantom{abcd} $k_{11}=1,k_{22}=1$}
    \includegraphics[width=0.25\textwidth]{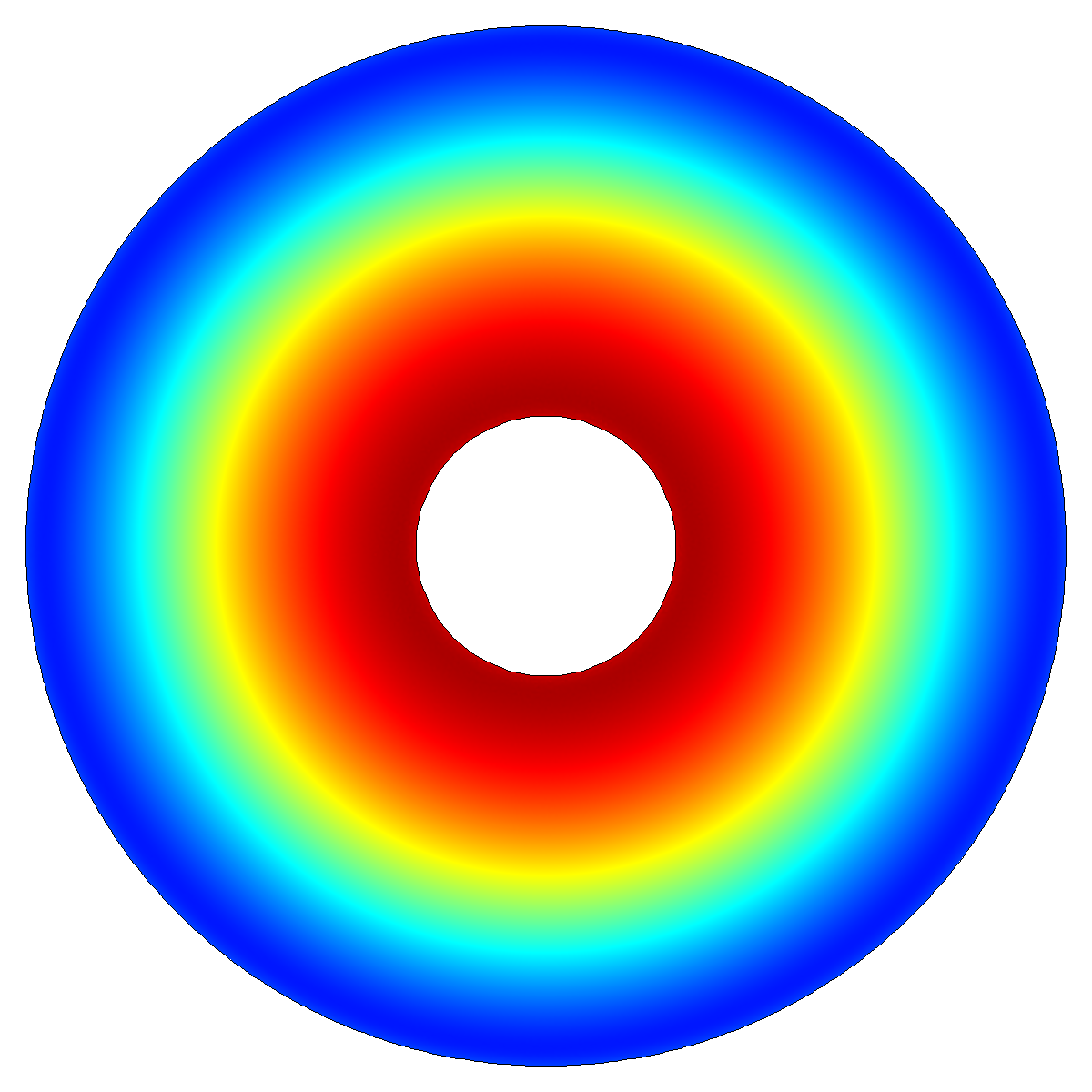}
    \includegraphics[width=0.25\textwidth]{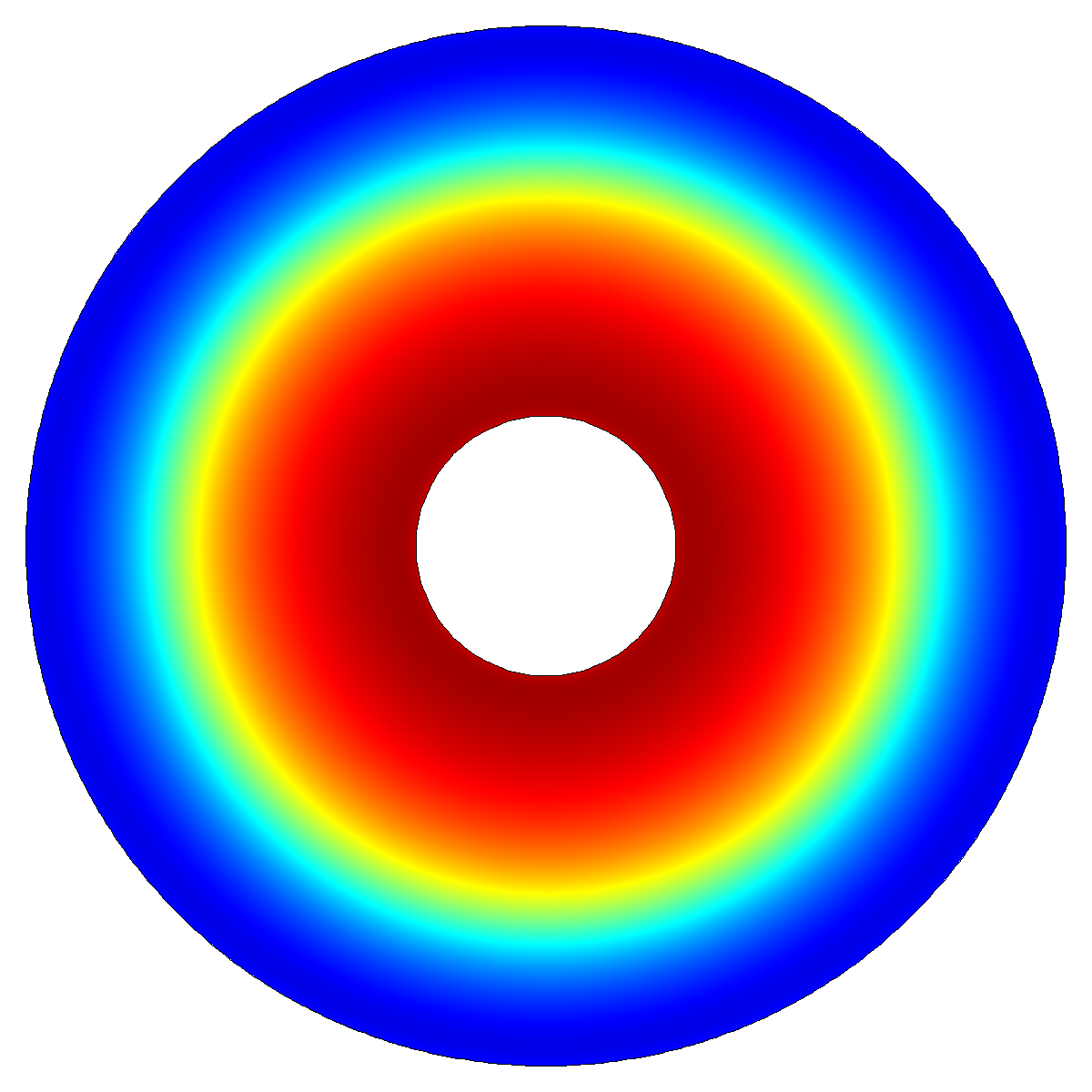}
    \includegraphics[width=0.25\textwidth]{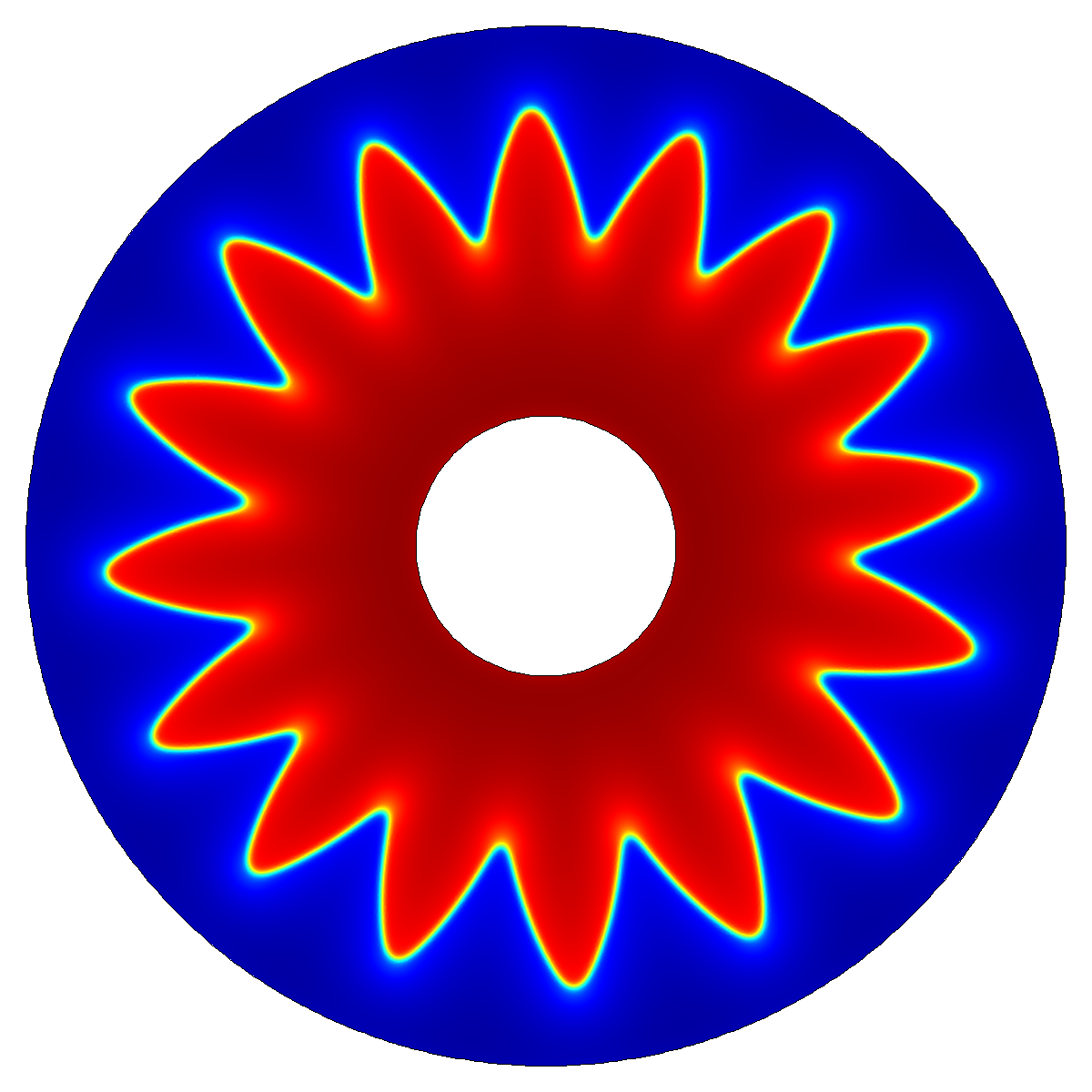}\\
        \hspace{0.2in} (a) \hspace{1.5in} (b) \hspace{1.5in} (c)\\
    \rotatebox{90}{\small \phantom{abcd} $k_{11}=0.1,k_{22}=0.1$}
    \includegraphics[width=0.25\textwidth]{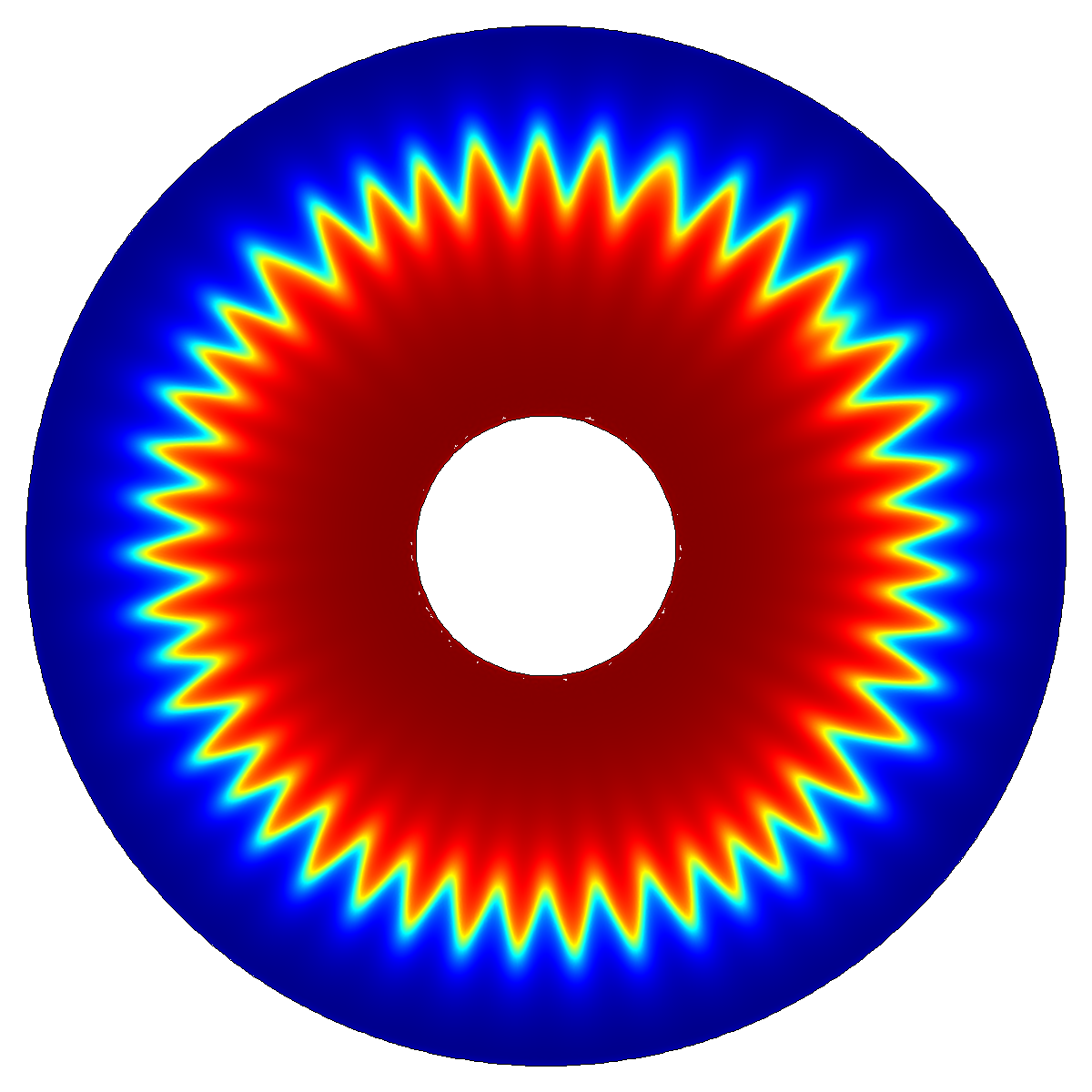}
    \includegraphics[width=0.25\textwidth]{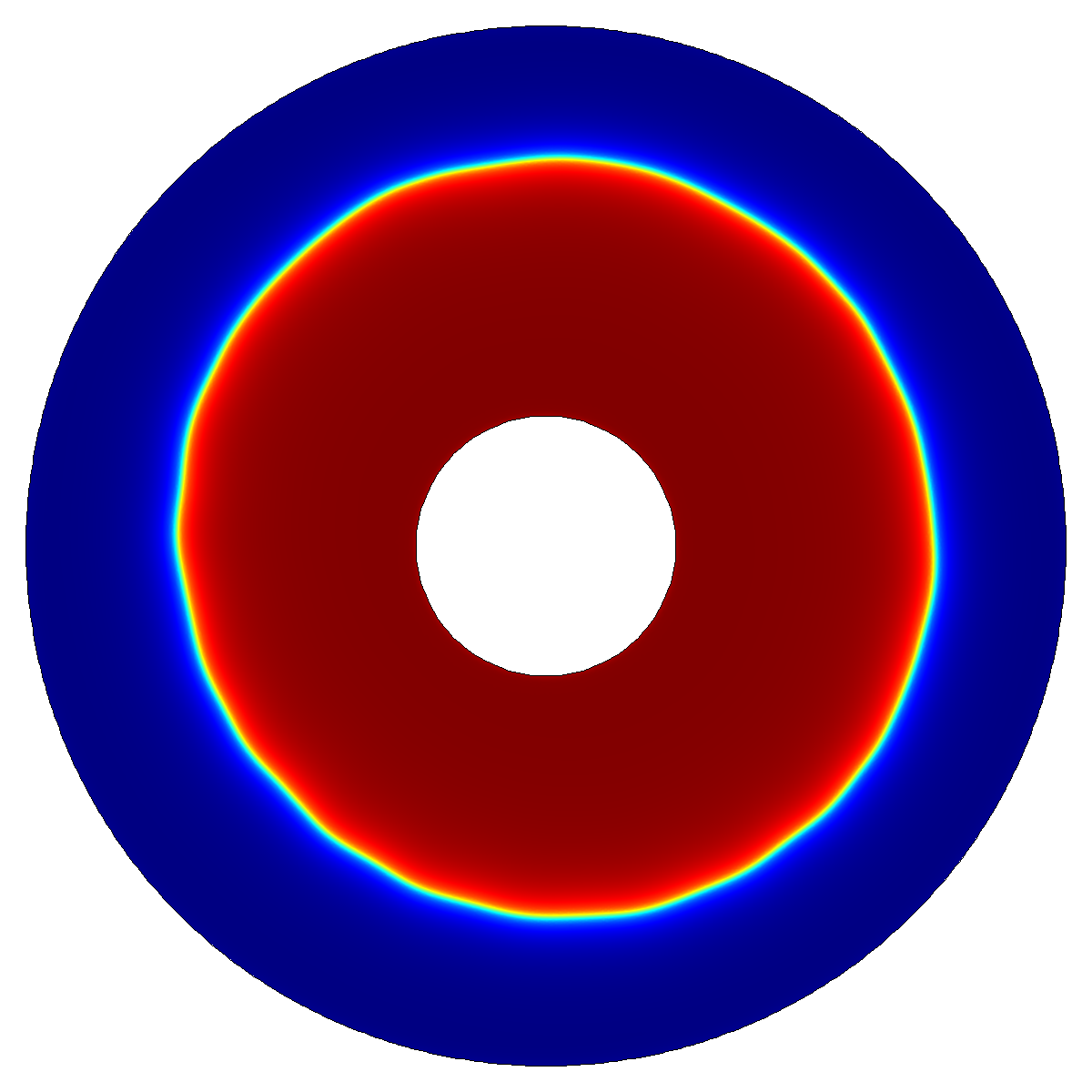}
    \includegraphics[width=0.25\textwidth]{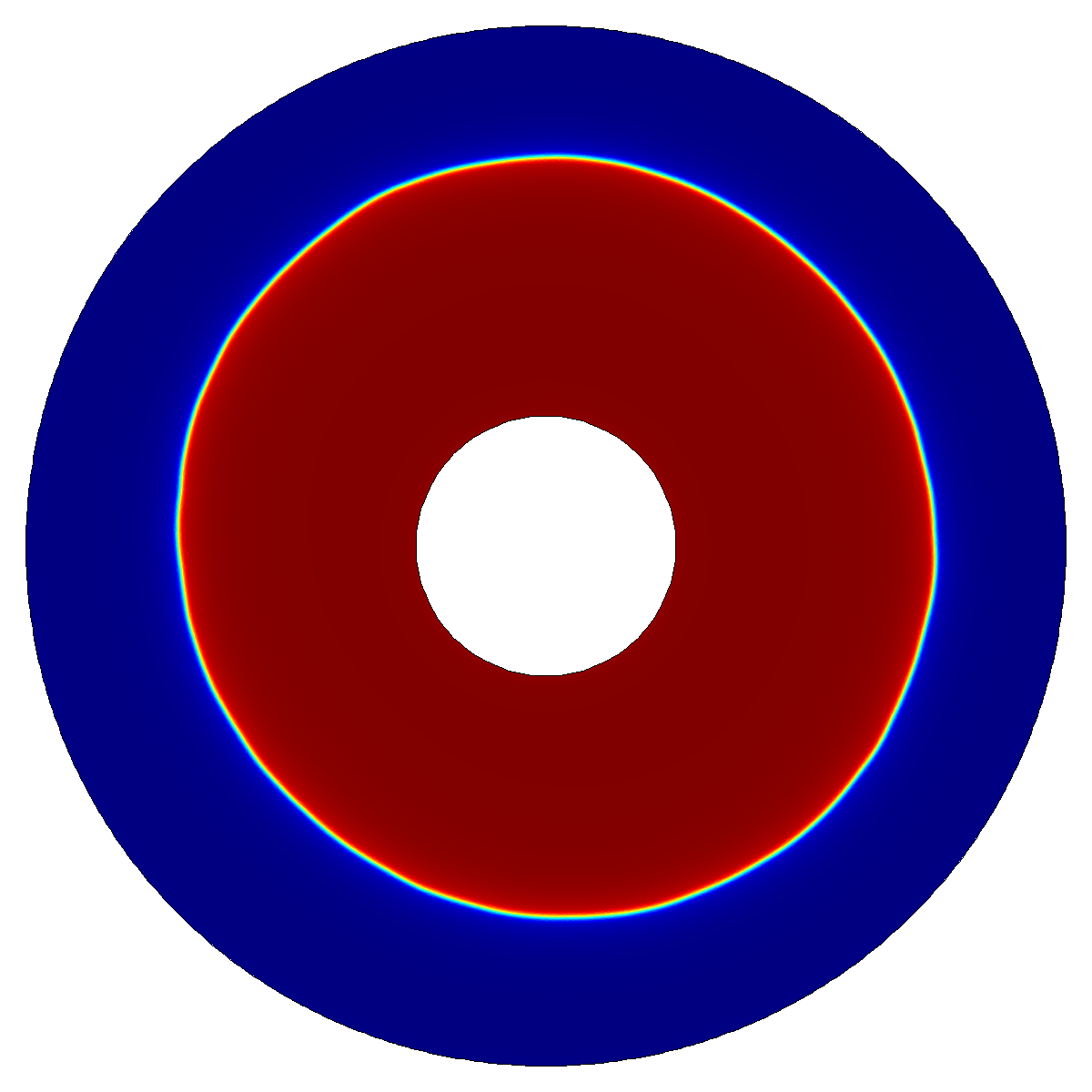}\\
        \hspace{0.2in} (d) \hspace{1.5in} (e) \hspace{1.5in} (f)\\
    \rotatebox{90}{\small \phantom{abcd} $k_{11}=1,k_{22}=0.1$}
    \includegraphics[width=0.25\textwidth]{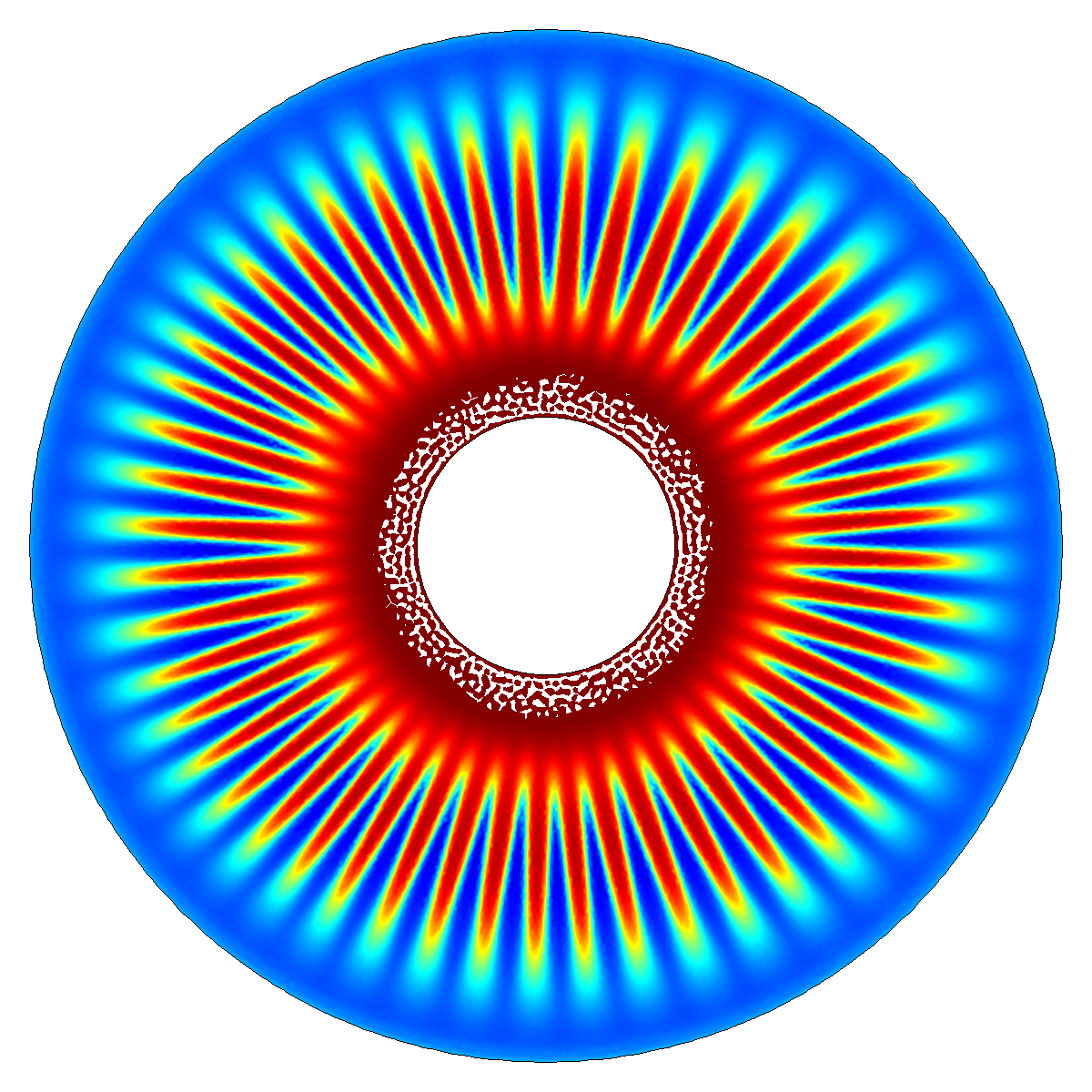}
    \includegraphics[width=0.25\textwidth]{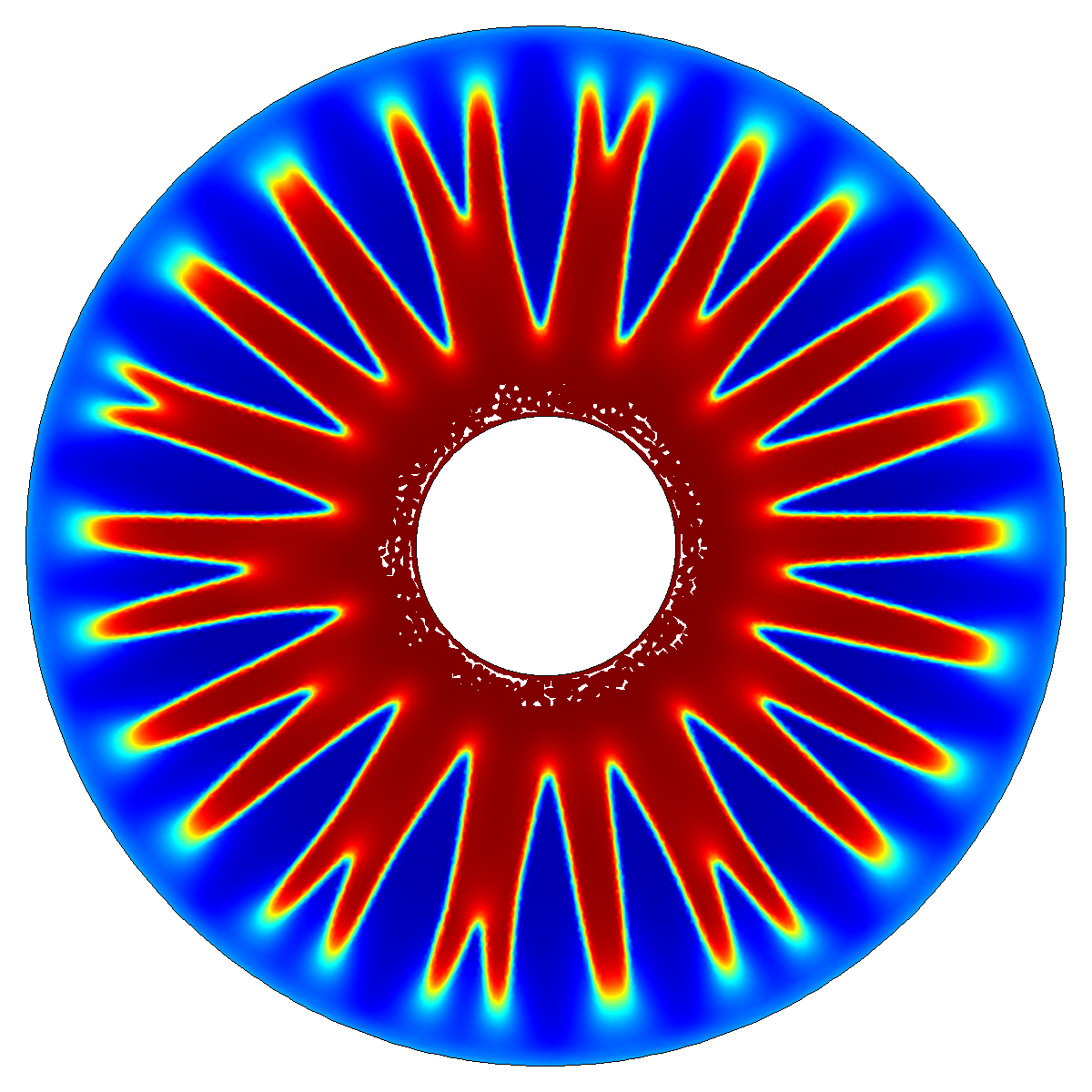}
    \includegraphics[width=0.25\textwidth]{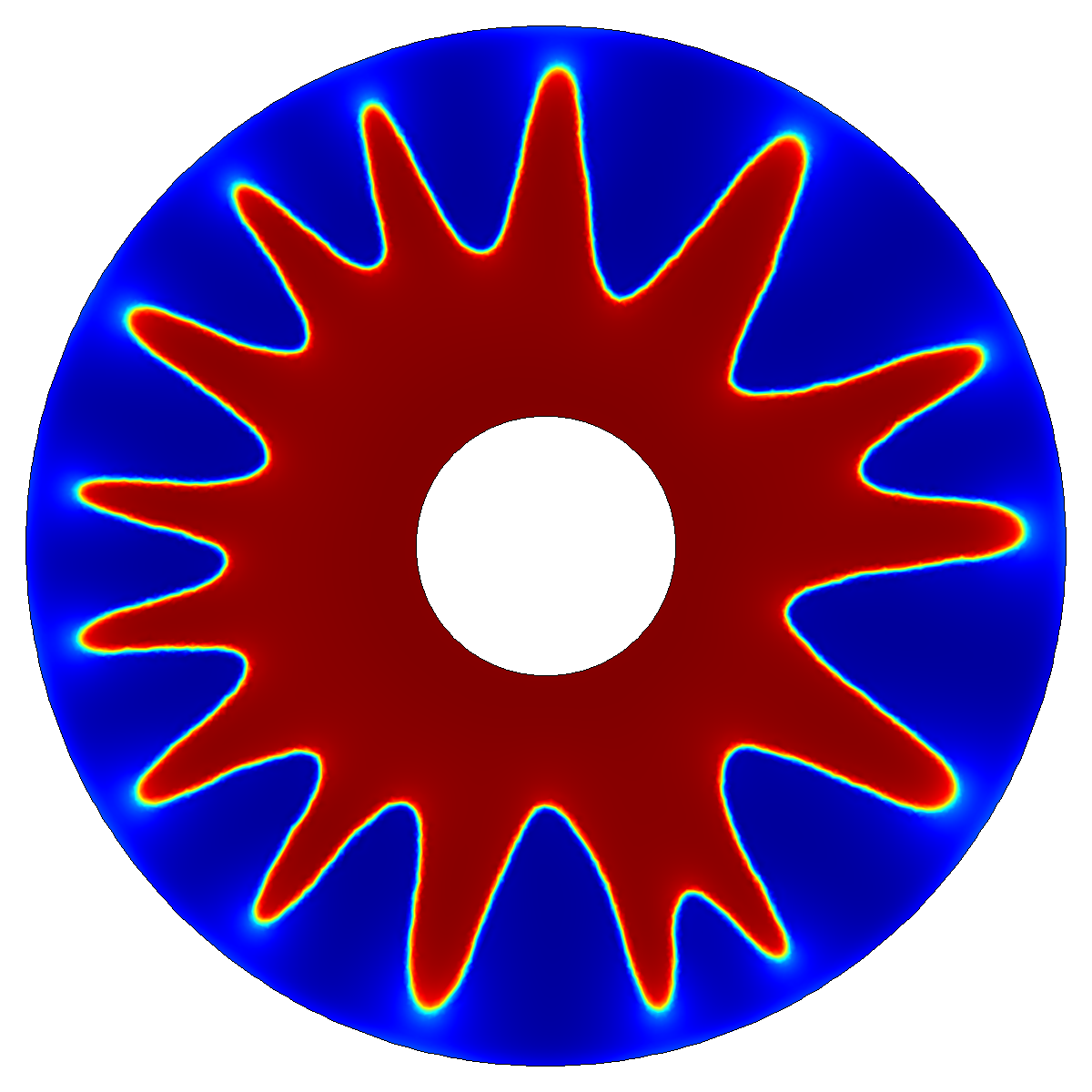}\\
       \hspace{0.2in} (g) \hspace{1.5in} (h) \hspace{1.5in} (i)\\
    \caption{Designs for a cylindrical reactor with a source of the first species at the inner edge and a sink for the second species at the outer edge.  The parameters are in (\ref{eq:param1}) except as noted in the rows and columns of the figure.  Further, $k_{12}=10^{-2} \times k_{11}, k_{21}=10^{-2}\times k_{22}$. \label{fig:cyl}}
\end{figure}

Many reactors designed for thermochemical conversion devices implement a cylindrical ceramic structure that allow for even heating and easy transport of recant gas. Thus, for the second example we look at an annular structure where the inner edge with $r=0.2$ is held as the source of the first chemical species ($\partial_1 \Omega$ where $u_1=1$) and the outer at $r=1$ is set as a sink for the second  ($\partial_2 \Omega$ where $u_2=0$).  We consider the same parameters as (\ref{eq:param1}). The resulting design is shown in Figure \ref{fig:cyl}(b). The first species enters from the inside, reacts and converts to the second species which exits from the outside. Thus, we see much of the first material on the inside and the second on the outside. Further, to enable sufficient reaction, the interface region is graded.  If we decrease the phase field coefficients by an order of magnitude, we obtain the design in Figure \ref{fig:cyl}(a) where the mixed region increases as the penalty for deviating from the pure materials is reduced. On the other hand, increasing the phase field coefficients by an order of magnitude yields the design in Figure \ref{fig:cyl}(c). Indeed, here, the penalty for deviation from the pure phases increases and therefore the interface becomes corrugated allowing sufficient reaction.

The second row of Figure \ref{fig:cyl} show the analogous result when the diffusivity is reduced by an order of magnitude. Transport is now harder compared to the reaction, and therefore nearly pure phases dominate to ensure transport and complex interfaces are avoid due to the phase field.  Again, increasing the phase field parameters promotes pure phases. The final row of Figure \ref{fig:cyl} show the results for unequal conductivity. Since the transport of first species is easier, material 1 forms long arms to transport the first species to close to the outlet where the reaction takes. Further, increasing the phase field parameters promotes pure phases and leads to fewer arms.

\subsection{Periodic cellular reactor}
\begin{figure}
\centering     
    \includegraphics[width=0.3\textwidth]{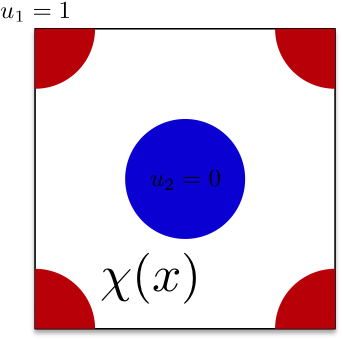} \hspace{0.25in}
    \includegraphics[width=0.3\textwidth]{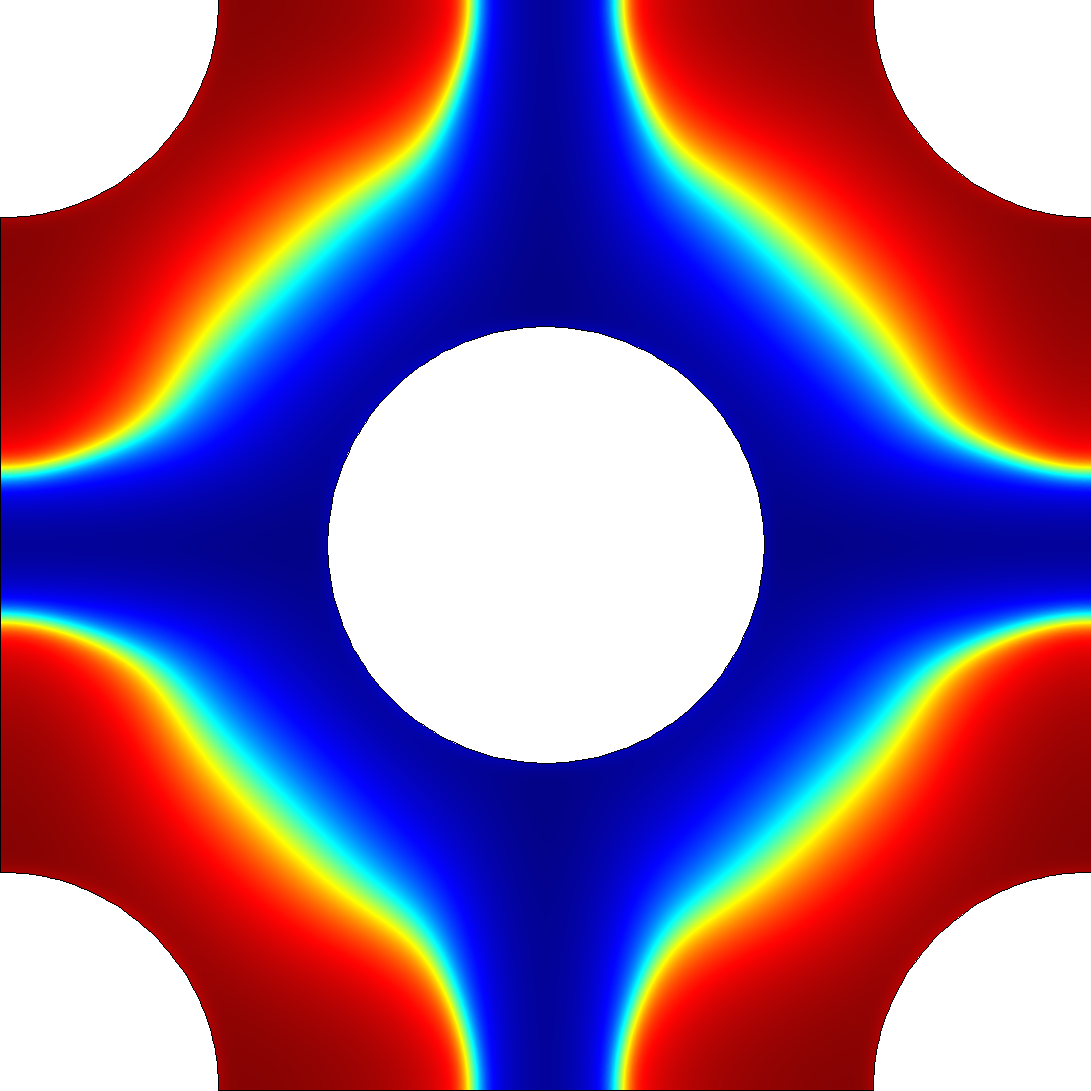}\\
    \hspace{-0.3in} (a) \hspace{1.7in} (b)\\
    \includegraphics[width=0.45\textwidth]{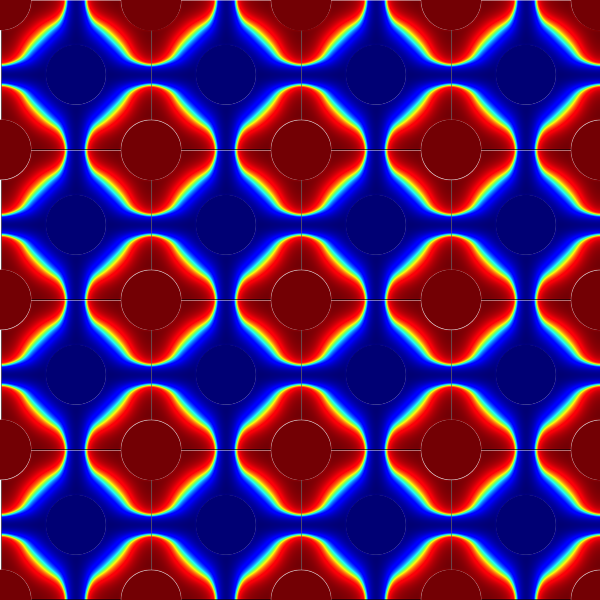}\\\
    (c)
    \caption{Periodic square distribution of circular sources and sinks. (a) Unit domain,  (b) Resulting design on the unit domain, (c) Periodic design.    \label{fig:periodic}}
\end{figure}

It is common to construct reactors as a periodic tubular array where the first species enter the reactor through one set of tubes while the second species is extracted from the reactor with a different set of tubes. Looking at a cross-section, one sees a square array of inlets and a square array of outlets. This motivates our next example where the reactor is taken to be periodic with the unit cell shown in Figure \ref{fig:periodic}(a). The source is at the corners of the cell while the outlet is at the center. We look for a periodic design to optimize the transport as before.  The resulting unit design for the parameters shown in (\ref{eq:param1}) is shown in Figure \ref{fig:periodic}(b).  It is repeated periodically in Figure \ref{fig:periodic}(c).

\section*{Acknowledgement}
It is a pleasure to acknowledge many interesting discussions with Sossina M. Haile and Robert V. Kohn. We gratefully acknowledge the financial support of the National Science Foundation through the PIRE grant  OISE-0967140.

\bibliography{mybib.bib}{}
\bibliographystyle{plain}

\end{document}